\documentclass[12pt]{amsart}
\usepackage{amssymb,amsthm,amsmath,amstext}
\usepackage{mathrsfs}  
\usepackage{bm}        
\usepackage{mathtools} 
\usepackage{fullpage}
\usepackage{graphicx}
\usepackage[all]{xy}

\mathchardef\mhyphen="2D

\theoremstyle{plain}
\newtheorem{theorem}{Theorem}[section]
\newtheorem{prop}[theorem]{Proposition}
\newtheorem{lemma}[theorem]{Lemma}

\newtheorem{hyp}[theorem]{Hypothesis}
\newtheorem{cor}[theorem]{Corollary}

\newtheorem*{conj1}{Conjecture 1}
\newtheorem*{conj2}{Conjecture 2}
\newtheorem*{nonu-theorem}{Theorem}

\theoremstyle{definition}

\theoremstyle{remark}

\newcommand{\sheaf}[1]{\mathscr{#1}}

\newcommand{\OO}{\mc{O}}

\renewcommand{\AA}{\sheaf{A}}
\newcommand{\BB}{\sheaf{B}}
\newcommand{\PP}{\sheaf{P}}
\newcommand{\XX}{\sheaf{X}}

\newcommand{\DD}{\sheaf{D}}

\newcommand{\UU}{\sheaf{U}}
\newcommand{\mc}[1]{\mathcal #1}

\newcommand{\Z}{\mathbb Z}

\newcommand{\G}{\mathbb G}

\usepackage{hyperref}
\usepackage{color}
\hypersetup{colorlinks=true,linkcolor=blue,anchorcolor=blue,citecolor=blue,letterpaper=true}

\DeclareFontFamily{U}{wncy}{}
    \DeclareFontShape{U}{wncy}{m}{n}{<->wncyr10}{}
    \DeclareSymbolFont{mcy}{U}{wncy}{m}{n}
    \DeclareMathSymbol{\Sha}{\mathord}{mcy}{"58}

\begin{document}

\title[Local global principle]
{ Local-global principle for    groups    of type $A_n$ over semi global fields} 
 
\author[Suresh]{V.\ Suresh}
\address{Department of Mathematics  \\ %
Emory University \\ %
400 Dowman Drive~NE \\ %
Atlanta, GA 30322, USA}
\email{suresh.venapally@emory.edu}

\begin{abstract}   Let   $F$  be the function field of a curve over  a complete discretely valued field $K$.
Let $G$ be a  semisimple simply connected linear algebraic group over $F$ of   type $A_n$. 
We  give a description of  the obstruction to local global principle for  principal homogeneous  spaces under $G$ 
over $F$  with respect to 
discrete valuations  of $F$ in terms of $R$-equivalence  classes of $G$ over some suitable over fields. Using this 
description  we prove  that  this  obstruction  vanishes   under some conditions on the residue field $K$. 
\end{abstract} 

\maketitle

\section{Introduction} 

Let $F$ be a field and $\Omega_F$ a set of discrete valuations of $F$. For $\nu \in F$, let $F_\nu$ be the completion of $F$ at $\nu$. 
Let $Z$ be a variety over $F$. We say that $Z$ satisfies a {\it local-global principle} with respect to $\Omega_F$ if 
$Z(F) \neq \emptyset$ if and only if $Z(F_\nu) \neq \emptyset$ for all $\nu \in \Omega_F$. 
 Let $K$ be a complete discretely valued field and $F$ the function field of a curve over $K$. Let $\Omega_F$ be the 
 set of diviosrial discrete valuations of $F$.
 Let $G$ be a  connected linear algebraic group over $F$.  The study of local-global principle for 
 homogenous  spaces of $G$ over $F$ with respect to $\Omega_F$ 
  were studied extensively in last 15 years,  beginning with patching
 techniques developed by Harbater, Hartmann and Krashen (\cite{HHK1},\cite{HHK3}).  
  The patching techniques give a local global principle for homogeneous spaces under
  rational groups with respect to certain over fields of $F$ (\cite[Theorem 3.7]{HHK1}).
   Since then  we are interested in the study of local-global principle  with respect to $\Omega_F$. 
   Considerable progress has been for function fields of curves over $p$-adic fields (\cite{CTPS1}, \cite{preeti},  \cite{PPS}, \cite{Wu}, 
   \cite{PS2022}).
   In (\cite{CTPS1}), Colliot-Th\'el\`ene, Parimala and Suresh made the following conjectures. 
   
   \begin{conj1} Let $F$ be the function field of a $p$-adic curve and $G$ a connected linear algebraic group over $F$. 
   Let $Z$ be a projective homogenous space under $G$. Then $Z$  satisfies local-global principle with respect to $\Omega_F$.  
   \end{conj1} 
   
   \begin{conj2} Let $F$ be the function field of a $p$-adic curve and $G$ a    semisimple simply connected algebraic group over $F$. 
   Let $Z$ be a principal  homogenous space under $G$. 
   Then $Z$  satisfies local-global principle with respect to $\Omega_F$.  
   \end{conj2}

The conjecture 1 and 2  have  been settled affirmatively for all classical groups with some assumption on $p$ (\cite{preeti}, 
\cite{Hu}, \cite{Wu}, \cite{PS2022}).
In fact recently Gille and Parimala (\cite{GP}) proved that if $F$ is the function field of a curve over   a  complete discrete valued field 
$K$ and  $G$  a connected linear algebraic group over $F$, then projective homogeneous spaces under $G$ over $F$ 
satisfy local global principle with respect to $\Omega_F$ (with some   assumptions on the characteristic of the residue field of $K$). 

In contrast to projective  homogenous spaces, there are  examples of connected linear algebraic groups 
 over the function fields $F$  of curves over 
complete discretely values fields $K$ and principal homogeneous spaces  for which local-global principle with respect to $\Omega_F$ 
fail.  The first such an example is  given in (\cite{CTPS2}) and the group in this example is torus and the residue field of $K$ is 
algebraically closed. However if we restrict to semisimple  simply connected groups, examples of such groups for which 
the  local-global principle with respect to $\Omega_F$  for principal homogenous spaces fails are given in (\cite{CTHHKPS2}).
In these examples the residue fields  of $K$ are of cohomological dimension at least 4. It is natural to ask what happens if 
the cohomological dimension of the residue field of $K$ is less than 4. If the cohomological dimension of the residue field of $K$
is zero, then cohomological dimension of $F$ is 2  (e.g. $F = C((t))(X)$  where $C$ is algebraically closed field and
$X$ a curve over $C((t))$)  and hence by Serre's conjecture, they are no nontrivial principal homogenous spaces
over $F$ under semisimple simply connected linear algebraic groups (\cite{BP}, \cite[Theorem 1.5]{CTGP}). 

 In this paper, we consider complete discrete valued field with  cohomological dimension of the residue field 1 or 2.
 Let  $F$ be a field, $\Omega_F$ a set of  divisorial discrete valuations of $F$ and $G$ a connected linear algebraic group over $F$.
Let 
$$\Sha_{div}(F, G) = ker(H^1(F, G) \to \prod_{\nu \in \Omega_F} H^1(F_\nu, G)).$$

Let $G$ be a   simply connected group    of type  $A_n$.
We say that  char$(\kappa)$ 
is {\it good}  for $G$  if   either char$(\kappa) = 0$ or if  $G$ is of type $^1A_n$, then 
$(n+1)$ is coprime to char$(\kappa)$ and if  $G$ is of type $^2A_n$, then 
$2(n+1)$ is coprime to char$(\kappa)$. 

For a field $L$ and   integer $m \geq 2$,  the {\it  m-cohomological dimension} of a field $L$, denoted by
$cd_m(L)$,  is defined as the maximum of the $p$-cohomological dimensions  $cd_p(L)$ (\cite[I.3.1]{SerreGC})
 for all primes $p$ dividing $m$. 

  We prove the following
 
 \begin{theorem} (\ref{an-cd1}) Let $K$ be a complete discretely  valued  field  with residue field $\kappa$. 
  Let $F$ be the function field of a curve over $K$.  Let $G$ be a semisimple simply connected group  over $F$ of type
  $A_n$.  If   cd$_{n+1}(\kappa) \leq 1$ and 
  char$(\kappa)$ is good for $G$, then $\Sha_{div}(F, G) = \{ 1 \}$. 
 \end{theorem} 

 \begin{theorem} (\ref{an-cd2}) Let $K$ be a complete discretely  valued  field  with residue field $\kappa$. 
  Let $F$ be the function field of a curve over $K$.  Let $G$ be a semisimple simply connected group over $K$
   of type $A_n$. If   cd$_{n+1}(\kappa) \leq 2$ and 
  char$(\kappa)$ is good for $G$, then $\Sha_{div}(F, G) = \{ 1 \}$.
  \end{theorem} 
 
 For general residue field, we prove the following. 
 
 \begin{theorem} (\ref{good-reduction-sl11}, \ref{good-reduction-su})
 Let $K$ be a complete discretely  valued  field  with residue field $\kappa$. 
  Let $F$ be the function field of a curve  $X$ over $K$. Suppose that $X$ has a good reduction. 
    Let $G$ be a semisimple simply connected group over $K$
   of type $A_n$.   If    
  char$(\kappa)$ is good for $G$, then 
   $\Sha_{div}(F, G) = \{ 1 \}$.
    \end{theorem}

One of the steps in the proof of the above theorems is the   following. 

 \begin{theorem}(\ref{sk1-requiv}, \ref{su}) Let $K$ be a complete discretely  valued  field  with residue field $\kappa$. 
  Let $F$ be the function field of a curve over $K$.  Let $G$ be a semisimple simply connected group  over $F$ of type
  $A_n$.  Then there exists a regular proper model $\XX$ of $F$ and a finite set  $\PP$ of closed points  
  of  $\XX$ such that 
   $$\Sha_{div}(F,  G)  \simeq \prod_{U \in \UU_0} G(F_{U}) /R \,\backslash \prod_{\wp \in \BB_0} G(F_\wp)/R 
  \,/ \prod_{P \in \PP_0} G(F_{P}) / R.$$ 
where  $\UU_0$ is the set of irreducible components of $X \setminus \PP_0$ and 
$\BB_0$ is the set of branches with respect to $\PP_0$. 

 \end{theorem}

We now give a brief description of the structure of the manuscript. 
In \S\ref{prel}, we begin by recalling  a few definitions and facts about central simple algebras and involutions. 
Let $K$ be a complete discretely  valued  field  with valuation ring $T$ and  residue field $\kappa$. 
  Let $F$ be the function field of a curve over $K$.  Let $\XX$ be a  regular proper model of 
  $F$ over $T$ and $X_0$ the special fibre.   For any $x \in X_0$, let $F_x$ denote the field of fractions of 
  the completion of the local ring at $x$. 
   In \S\ref{gen}  we give some sufficient  conditions on the group $G$ under which the obstructions to local global principle
  with respect to divisorial discrete valuations and with respect to the  over fields  $F_x$, $x \in X_0$. 
  In \S\ref{isotropic}, under some conditions on $G$, we give a description of the obstruction to the local global principle 
  with respect to $F_x$, $x \in X$, in terms of $R$-equivalence classes of $G$ over certain suitable over fields of $F$. 
 
  Let $R$ be a two dimensional  complete regular local ring  with field of fractions $F$, residue field $\kappa$ and maximal
  ideal $(\pi, \delta)$. 
  In \S\ref{field-extns} we give a description of   finite extensions of $F$ of degree coprime to char$(\kappa)$ which are unramified 
  on $R$ except possibly at $(\pi)$ and $(\delta)$.   Let $A$ be a central simple algebra over $F$ which is unramified 
  on $R$ except possibly at  $(\pi)$ and $(\delta)$. Suppose that the period of $A$ is coprime to char$(\kappa)$.
  In \S\ref{csa-2dim}, we show  that  index$(A) = $ index$(A\otimes F_\pi)$. We also show that
   an element in $F^*$ which is supported only along $(\pi)$ and $(\delta)$ is a reduced norm from $A$
  if and only if it is a reduced norm from $A \otimes F_\pi$.  In \S\ref{sk1d-2dim}, we show that the natural map 
  $SK_1(A) \to SK_1(A\otimes F_\pi)$ is onto. If $A$ has involution of second kind, then 
  in \S\ref{sk1ud-2dim}, $SUK_1(A,\tau) \to SUK_1(A\otimes F_\pi, \tau)$ is onto.
  In \S\ref{sk1d-dvr}, we recall description of $SK_1(A)$ for a central simple algebra 
  over a complete discretely  valued field.   In \S\ref{sk1ud-dvr}, we recall description of $SUK_1(A, \tau)$ 
  for a central simple algebras with involutions
  over a complete discretely  valued field.   
  
  Let $F$ be the function field of a curve over a complete discretely valued field $K$. 
  Let $A$ be a central simple algebra over $F$. 
  In  In \S\ref{type-inner},   we give a description of the obstruction to the local global principle 
  with respect to  divisorial discrete valuations of $F$  in terms of $R$-equivalence classes of $SL_1(A)$ 
  over certain suitable over fields of $F$. If $A$ has an involution $\tau$ of second kind, then in   \S\ref{lgp-unitary}, 
  we prove a local-global principle for unitary groups and 
  in  \S\ref{type-outer},     we give a description of the obstruction to the local global principle 
  with respect to  divisorial discrete valuations of $F$  in terms of $R$-equivalence classes of $SU(A, \tau)$. 
  In   \S\ref{cd1} and   \S\ref{cd2}, we give the proofs of our main results.

  \section{Preliminaries }
  \label{prel}
  
  In this section we recall some basic facts about central simple algebras and involutions.  
  For further details  we refer to (\cite{Draxl} and \cite{reiner}). 
     
  Let $K$ be a field and $A$ a central simple algebra over $K$. Let $Nrd: A \to K$ be the reduced norm
  ans   $SL_1(A) = \{ x \in A \mid Nrd(x) = 1\}$.  
  The exact sequence of algebraic groups 
  $$1 \to SL_1(A)\to GL_1(A) \to \G_m \to 1$$
  induces a long exact sequence of cohomology sets 
  $$ A^* \to K^* \to H^1(K, SL_1(A)) \to H^1(K, GL_1(A)).$$
  Since $H^1(K, GL_1(A)) = \{ 1 \}$  (cf., \cite[Theorem 29.2]{KMRT})  and $A^* \to K^*$ is the reduced norm 
     we have 
  $$H^1(K, SL_1(A)) \simeq K^*/Nrd(A^*).$$
 Let $A' = M_n(A)$. 
 Since $Nrd(A^*) = Nrd(A^{'*})$ (cf., \cite[Lemma 2.6.4]{GS}),  the map  
 $$ H^1(K, SL_1(A)) \to H^1(K, SL_1(M_n(A))$$ induced by the inclusion $SL_1(A) \to SL_1(M_n(A))$ is bijective.

   Two central simple  algebras  $A$ and $B$ over a field $K$ are  {\it Brauer equivalent } if 
  $M_n(A) \simeq M_m(B)$ for some $n$ and $m$. Throughout this paper we write $A = B$ if 
  $A$ and $B$ are Brauer equivalent.   
  
    Let $SK_1(A) = SL_1(A)/[A^* : A^*]$. 
  If $A = B$, then $SK_1(A) \simeq SK_1(B)$ (\cite[p.155]{Draxl}).

   Let $E/K$ be a  cyclic extension with $\sigma$ a generator of Gal$(E/K)$. 
   For $a \in K^*$, let   $(E, \sigma, a)$ denotes the cyclic algebra given by 
  $$(E, \sigma, a) = E \oplus Ex  \oplus \cdots \oplus Ex^{n-1} $$  with $x^n = a \in K$ and 
  $x\lambda = \sigma(\lambda) x$ for all $\lambda \in E$. 
  If $L/K$ is a extension, then $E\otimes_K L = \prod E'$ for some cyclic extension $E'/L$ and
  $\sigma$ induces a generator $\sigma'$ of Gal$(E'/L)$.
  We denote the algebra $(E', \sigma', a)$ by $(E, \sigma, a) \otimes L$. 
  
  Suppose that $K$ is a  discrete valued field with residue field $\kappa$ and $R$ the valuation ring. 
  Let $\hat{K}$ be the completion of  $K$. 
  We say  that a central simple algebra $A$ over $K$
     is {\it unramified} if there exists an Azuamaya $R$-algebra
  $\AA$ such that $\AA  \otimes _R K \simeq A$.  
  Suppose that ind$(A)$ is coprime to  char$(\kappa)$.
  Then there exists a cyclic unramified  extension  $E/K$ of degree $n$  such that 
  $A$ is Brauer equivalent to $ A_0 + (E, \sigma, \pi)$ (cf., \cite[Lemma 4.1]{PPS}), where  $A_0$ is an unramified 
  central simple algebra over $K$, $\sigma$ is a generator of Gal$(E/K)$ and $\pi \in K$ a parameter. 
  
  Let $D$ be a (finite dimensional) central division algebra over $K$. Suppose that $D\otimes_K \hat{K}$ is division. 
  Then the valuation on $K$ extends to a valuation on $D$. Let $\Lambda = \{ x \in A \mid Nrd(x) \in R \}$.
  Then $\Lambda$ is the unique maximal $R$-order in $D$.  Further $\Lambda$ has a unique 2-sided maximal ideal 
  $m_D$ and $\bar{D} = \Lambda/m_D$ is a division algebra.  Suppose ind$(D)$ is coprime to char$(\kappa)$.
  Then we have $D = D_0 + (E, \sigma, \pi)$ for some unramified central division algebra  $D_0$ and unramified cyclic extension 
  $E/K$.  In this case $Z(\bar{D})$ is the residue field of $E$. In particular $Z(\bar{D})/\kappa$ is a cyclic extension. 
  Further $\bar{D}$ is Brauer equivalent to $\bar{D}_0 \otimes Z(\bar{D})$. 
  
  We record here the following well known result. 
  
  \begin{prop} 
  \label{cyclic} 
  Let $K$ be a complete discretely valued field with residue field $\kappa$.
  Let $A$ be a central simple  algebra over $K$ of period coprime to char$(\kappa)$. 
  If there exists a totally ramified extension $L/K$  which splits $A$, then  
  there exists a cyclic extension $E/K$ and a parameter $\pi \in K$ such that  
  $A   = (E, \sigma, \pi)$. 
  \end{prop}
  
  \begin{proof} Since period of $A$ is coprime to char$(\kappa)$, there exists a  an unramified central simple 
  algebra $A_0$ over $K$ and cyclic extension $E/K$ with a generator $\sigma$ of Gal$(E/K)$ such that 
  $A =A_0 + (E,\sigma, \pi)$.  Suppose there exists a totally ramified extension $L/K$ such that 
  $A \otimes L$ is split.   Then $A \otimes E$ is split (cf. \cite[Lemma 4.3]{PPS}). In particular 
  $A_0 \otimes E$ is split. Hence, $A = (E, \sigma, u\pi)$ for some   $u \in K^*$ a unit in the valuation ring of $K$
   (cf. \cite[Lemma 4.4]{PPS}. 
  \end{proof}
    
  Let $\XX$ be a normal integral scheme with function field $F$. 
  Let $x \in \XX$ be a codimension one point. Then the local ring $\OO_{\XX, x}$ at $x$ is a discrete valuation ring. 
 Let $A$ be a central simple algebra over $F$. We say that $A$ is {\it unramified at x} if $A$ is unramified at the discrete 
 valuation ring $\OO_{\XX, x}$. We say that $A$ is {\it unramified on $\XX$} if $A$ is unramified at every 
 codimension one point of $\XX$. 
 
 Let $G$ be a connected reductive linear algebraic group over $F$. 
Let $x$ be   a codimension one point   of $\XX$. 
 We say that  $G$ is {\it unramified}  at   $x$   if there is a 
 reductive group defined over the local ring $\OO_{\XX, x}$  with the generic fibre isomorphic to $G$. 
 If $G$ is unramified at evert codimension one point of $\XX$, then we say that 
 $G$ is unramified on $\XX$. 
 Similarly a $G$-torsor $\zeta$ is  {\it unramified}  at   $x$   if  $G$ is unramified at $x$ and the torsor is defined  over the 
 local ring at $x$.  If $\zeta$ is unramified at every codimension one point of $\XX$, then we say that
 $\zeta$ is unramified on $\XX$.  If $\XX = Spec(R)$ for some normal domain, we say that 
 $G$ (or a $G$-torsor) is unramified on $R$ if  it is unramified on $\XX$. 
 The union of all codimension one points of $\XX$ where $G$ is not unramified is called the {\it 
 ramification locus} of $G$. 
 
 If a central simple algebra $A$ over $F$ is unramified on $\XX$, then the group $SL_1(A)$ is unramified on $\XX$. 
 
 Let $K_0$ be a field and $K/K_0$ a quadratic separable field  extension.  Let Let $A$ be a central simple 
 algebra over $K$ with a $K/K_0$-involution $\tau$. 
 Let 
 $$U(A, \tau) = \{ a \in A \mid \tau(a)a  = 1\}$$ 
 be the unitary group of $(A, \tau)$ and 
 $$SU(A, \tau)  = \{ a \in SU(A,\tau) \mid Nrd(a) = 1 \}$$ be the special unitary group $(A, \tau)$. 
 Let $R^1_{K/K_0} \G_m$ be the group of norm one elements in $K/K_0$. 
 For an extension $L/K_0$,   
 if $L\otimes_{K_0}K$ is not a  field, then 
 $L\otimes_{K_0}K \simeq L \times L$, $K \subset L$ and $A\otimes _{K_0} K \simeq (A\otimes_KL) \times 
 (A\otimes_KL)^{op}$ and  $SU(A, \tau)(L) = SL_1(A\otimes_KL)$ (\cite[p.346]{KMRT}).

 We  have the short exact sequence of algebraic groups  
 $$ 1 \to SU(A, \tau) \to U(A,\tau) \to R^1_{K/K_0}\G_m \to 1.$$
 
 Let  $K^{*1} =  \{ a \in K^* \mid N_{K/K_0}(a) = 1\}$ and 
 $H_A = \{ \theta \sigma(\theta)^{-1} \mid \theta \in Nrd(A) \} \subseteq L^{*1}$. 
 
Then the above exact sequence gives a long exact sequence of cohomology sets   
$$U(A, \tau)   \to    K^{*1}   \to  H^1(K_0, SU(A, \sigma))   \to  H^1(K_0, U(A, \sigma)), $$
 with  the image of $U(A, \tau)   \to    K^{*1} $  equal to $H_A$ (\cite[p.202]{KMRT}).  Hence we have the exact sequence of pointed sets 
 $$1 \to  K^{*1}/H_A \to H^1(K_0, SU(A, \sigma))   \to  H^1(K_0, U(A, \sigma) .   $$

  Let  $K_0$ be a complete discretely valued field with valuation ring $R_0$ and $K/K_0$ a separable quadratic field extension. 
  Let $R$ be the integral closure of $R_0$ in $R$. 
  Let $A$ be a central simple algebra over $K$ with a $K/K_0$-involution $\tau$.
  We say that $(A, \tau)$ is {\it unramified } if there exists an Azumaya algebra $\AA$ over $R$ and
  a $R/R_0$-involution $\tau_0$ such that $(\AA, \tau_0) \otimes _{R_0} K_0 \simeq (A,\tau)$.

  Suppose that $D$ is central division algebra over $K$ with a $K/K_0$-involution. 
  Let $\Lambda$ be the unique maximal $R$-order in $D$. Suppose that $D \neq K$ and 2ind$(D)$ coprime 
  to characteristic  of the residue field of $K_0$.
  Then there exists a parameter $\pi_D \in D$ such that $\tau(\pi_D) = \pi_D$.
  Further the inner automorphism of $D$ given by $\pi_D$ takes $\Lambda$ to $\Lambda$ 
  and hence induces an automorphism of $\bar{D}$. This induced automorphism restricted $Z(\bar{D})$
  is a generator of the Gal$(Z(\bar{D})/\kappa)$, where $\kappa$ is the residue field of $K$.

  Let $\XX_0$ be a normal integral scheme with function field $F_0$ and $F/F_0$ a separable quadratic extension. 
  Let $\XX$ be the normal closure of $\XX_0$ in $F$. 
  Let $x \in \XX_0$ be a codimension one point. 
  Then the local ring $\OO_{\XX, x}$ at $x$ is a discrete valuation ring.  There are at most two codimension one points 
  $x_1, x_2 \in \XX$ lying over $x$. Suppose that 
  Let $A$ be a central simple algebra over $F$ with a $F/F_0$-involution $\tau$. 
 We say that $(A, \tau)$ is {\it unramified at x} if $(A, \tau)$ is unramified at the discrete 
 valuation ring $\OO_{\XX_0, x}$. We say that $(A, \tau)$ is {\it unramified on $\XX_0$} if $(A, \tau)$ is unramified at every 
 codimension one point of $\XX_0$.

 Let $R_0$ be a  principal ideal domains  with field of fractions $F_0$ and $F = F_0(\sqrt{u})$ for some unit $u \in R_0$.
 Suppose 2 is a unit in $R_0$. 
   Let $R = R_0[\sqrt{u}]$ be the integral closure of
 $R$ in $L$.     
 
 \begin{prop}
 \label{purity}
 Let $A$ be a central simple algebra over $F$ with a $F/F_0$-involution $\tau$.
 Suppose that $(A, \tau)$ is unramified on $R_0$. 
 Then there exists an Azumaya algebra   $\AA$  over $R$  such that 
   $\AA  \otimes_{R} F  =  A$ and $\tau(\AA) = \AA$. 
   \end{prop}
 
 \begin{proof}  Let $\bar{ ~} : F \to F$ be the non trivial automorphism of  $F/F_0$. 
 Let $n = deg(A)$. 
 Then $\bar{R} =R$. Let $\AA_0 = M_n(R)$ and $\tau_0$ be the  $F/F_0$-involution on 
 $\AA_0$ given by $\tau_0( a_{ij}) = (\bar{a}_{ji})$. 
 Since $R/R_0$ is \'etale, the group  $ G = {\bf Aut}(\AA_0, \tau_0)$ is a reductive group over $R_0$.
 The set $H^1(F_0, G)$   classifies  isomorphism classes of pairs $(B, \sigma)$  of central simple algebras 
 $B$ over $F$  of degree $n$ together with $F/F_0$-involutions $\sigma$  (\cite[29.14]{KMRT} ).
 Similarly $H^1_{et}(R_0, G)$  classifies  isomorphism classes of pairs $(\BB, \sigma)$  of Azumaya algebras 
 $\BB$ over $R$  of degree $n$ together with $R/R_0$-involutions $\sigma$.
 
 Since $(A,\tau) \in H^1(F, G)$ is unramified on $R$ and $R$ is a principal ideal domain, by (\cite[Proposition 6.8]{CTS},see also 
 \cite[Corollary A.8]{GP08}), 
 there exists $(\AA', \tau') \in H^1_{et}(R, G)$ which maps to $(A, \tau)$
 Hence $(\AA', \tau') \otimes_{R_0} F_0 \simeq (A, \tau)$.
 Then  the image $\AA$ of $\AA'$ in $A$ has the required property.  
 \end{proof}

 Since $R/R_0$ is \'etale, the groups $U(\AA,  \tau)$ and $SU(\AA,   \tau)$
 are reductive groups over $R_0$.

 We have the exact sequence 
 $$ 1 \to SU(\AA, \tau)  \to  U(\AA, \tau) \to  R^1_{S_{\pi\delta}/R_{\pi\delta}}\G_m \to 1$$
of  algebraic group schemes over $R_0$, 
where $R^1_{R/R_0 }\G_m$ is defined as the kernel  of  
the norm map $  R_{R/R_0} \G_m \to  \G_m .$ 

We have the induced long  exact sequence 
$$ U(\AA, \tau_0)(R_0) \to R^1_{R/R_0 }\G_m (R_0 ) 
 \to H^1_{et}(R_0, SU(\AA, \tau)) \to H^1_{et}(R_0, U(\AA,\tau)).$$
 
 \begin{prop}
 \label{pid}
We have 
$$R^1_{R/R_0}\G_m (R_0)  \simeq \{ a\tau(a)^{-1}  \mid a \in R   \}. $$  
\end{prop} 

\begin{proof} Since $R$ and $R_0$ are principal ideal domains, the result follows from Hilbert 90.
\end{proof}

 We prove the following technical  lemma which is used later. 
  Let $K/K_0$ be a quadratic   extension  and $L/K$ a finite cyclic extension such that 
  $L/K_0$ is a dihedral extension.   Let $\sigma$ be a generator of Gal$(L/K)$. 
  Then there exists an automorphism $\tau_1$ of $L$ such that $\tau_1^2$ is identity   and $\sigma \tau_1 = \tau_1 \sigma^{-1}$. 
  Let $\tau_2  = \sigma  \tau_1   $.   Let $L_1 = L^{\tau_1}$. 
  
The following is extracted from  the proof of (\cite[Corollary 4.13]{Y1979}).
 
\begin{lemma} \label{conorms} Let $a \in L^*$. Suppose there exists $b \in L^*$ such 
   that $\sigma(a)a^{-1} = b\tau_2(b)^{-1}$. Then$a \in K^*L_1^*$. 
\end{lemma}
  
  \begin{proof}  Since  $\sigma(a)a^{-1} = b\tau_2(b)^{-1}$,  $\tau_2(\sigma(a)a^{-1})  (\sigma(a)a^{-1}) = 1$. 
  By  expanding this we get $\tau_2\sigma(a) \tau_2(a)^{-1}  \sigma(a)a^{-1} = 1$.
  Since  $\tau_2  \sigma = \sigma \tau_1  \sigma = \tau_1 $ and $\tau_2 = \sigma \tau_1$, we have
  $ \tau_1(a)\sigma\tau_1(a)^{-1} \sigma(a)a^{-1} = 1$. Hence $\sigma(a\tau_1(a)^{-1}) = a \tau_1(a)^{-1}$. 
  Since $\sigma$ is the generator of Gal$(L/K)$, $a\tau_1(a)^{-1} \in K^*$. 
  
  Let $c = a\tau_1(a)^{-1} \in K^*$. Since $N_{K/K_0}(c) = 1$ and $\tau_1$ restricted to $K$ is the nontrivial 
  automorphism of $K/K_0$, there exists $b \in K^*$ such that 
  $c = b \tau_1(b)^{-1}$.  Then $ab^{-1} = \tau_1(ab^{-1}$ and hence $ab^{-1} \in L_1$.
  Therefore $a \in K^*L_1^*$.  
  \end{proof}

\section{generalities} 
\label{gen}
 
 Let $T$ be a complete discrete valuation ring with residue field $\kappa$ and field of 
fraction $K$. Let $F$ be the function field of a curve over $K$. 
Let  $G$ be a connected reductive linear algebraic group  over $F$.  In this section we give some sufficient conditions under which 
the obstructions  to local global principle with respect to discrete valuations and with respect to patching over fields  coincide. 
For al unexplained notation, we refer to  
   (\cite{HHK1}, \cite{HHK3}).

 Let $\XX$ be a 
regular  proper model 
  of $F$ (\cite{Lip75}).  Let $x \in \XX$  be a point and 
$\hat{\OO}_{\XX, x}$  the completion of  the local ring  $\OO_{\XX, x}$ at $x$. 
Let  $F_x$ be  the field of fractions of  $\hat{\OO}_{\XX, x}$.
If $x \in \XX$ is a codimension one point, then $x$ gives a discrete valuation  $\nu_x$ on $F$.
A discrete valuation   $\nu$ on $F$ is called a   
{\it divisorial } discrete valuation if $\nu = \nu_x$
for some codimension one point of a    regular proper model   of $F$.
Let $\Omega_F$ be the set of all divisorial discrete valuations of $F$.

Let 
$$
\Sha_{div}(F, G) =  ker(H^1(F, G) \to   \prod_{\nu\in \Omega_F} H^1(F_\nu , G)) $$  
 
 and 
$$
\Sha_X(F, G) =  ker(H^1(F, G) \to   \prod_{ x \in X} H^1(F_x , G)) .$$

Let 
$$
\Sha(F, G) =  ker(H^1(F, G) \to   \prod_\nu H^1(F_\nu , G)), $$
where $\nu$ running over all the discrete valuations of $F$.
Since $\Sha(F, G) \subset \Sha_{div}(F, G)$ and 
  $\Sha_X(F, G) \subseteq \Sha(F, G)$ (\cite[Proposition 8.2]{HHK3}),
$\Sha_X(F, G) \subseteq \Sha_{div}(F, G)$.  

Let $P \in \XX$ be a closed point.   A discrete valuation $\nu$ of $F_P$ is called a {\it divisorial} discrete valuation of
$F_P$ if it is given by a codimension one point of a sequence of blow-ups of the local ring at $P$. 
Let $\Omega_P$ be the set of all divisorial discrete valuations of $F_P$ and 
$$
\Sha_{div}(F_P, G) =  ker(H^1(F_P, G) \to   \prod_{\nu\in \Omega_P} H^1(F_{P\nu} , G)). $$

Let $X$ be the special fibre of $\XX$ and $X_{(0)}$ the set of closed points of $X$. 
Since any divisorial discrete valuation of $F_P$ restricted to a $F$ is a divisorial discrete valuation of 
$F$,  the proof of   (see \cite[Proposition 8.4]{HHK3}) gives the following 

\begin{prop}
\label{exact-seq} With the notation as above we have the following  exact seqeunce
$$ 1 \to \Sha_X(F,G) \to \Sha_{div}(F,G) \to \prod_{P \in X_{(0)}} ^\prime \Sha_{div}(F_P, G)\to 1,$$
where $\displaystyle{ \prod_{P \in X_{(0)}} ^\prime} \Sha_{div}(F_P, G)$ is the subset of the $ \displaystyle{\prod_{P \in X_{(0)}}}
 \Sha_{div}(F_P, G)$
consisting of elements $(\zeta_P)$ with all but finitely many $\zeta_P$ are trivial.  
\end{prop}

 Suppose that   the union of the  ramification locus of $G$ and the closed fibre 
is a union of regular curves with normal crossings.
Let $P \in \XX$ be a closed point.  
Then there exist $\pi_P, \delta_P \in  \OO_{\XX, P}$  such that the maximal ideal 
of $\OO_{\XX, P}$ is generated by $\pi_P$ and $\delta_P$, and 
$G$ is unramified on $\OO_{\XX, P}$ except possibly at $(\pi_P)$ and $(\delta_P)$.  

Let 
$$H^1_{\pi_P\delta_P}(F_P, G) = \{ \zeta \in  H^1(F_P, G) \mid \zeta {\rm ~is ~unramified ~on ~ } \hat{\OO}_{\XX, P}
 {\rm ~except ~possibly ~at~}
(\pi_P) {\rm ~and ~} (\delta_P) \}$$ and
$$\Sha_{\pi_P\delta_P}(F_P, G) = ker(H^1_{\pi_P\delta_P}(F_P, G) \to \prod _{\nu \in \Omega_P}H^1(F_{P\nu}, G)).$$

\begin{hyp}
\label{hyp1}  
We say that 
$G$ satisfies the {\it local injectivity}  hypothesis   with respect to $\XX$
  if    for every closed point $P\in \XX$  the map  $H^1_{\pi_P\delta_P}(F_P, G) \to   \prod_{\nu\in \Omega_P}
   H^1(F_{P, \nu}, G)$ has trivial kernel.
   \end{hyp}

 \begin{theorem} (\cite[Proposition 8.4]{HHK3})  Let $F$, $G$ and $\XX$  be as above.   
If $G$ satisfies the local injectivity     hypothesis (\ref{hyp1}) with respect to $\XX$, then 
 $$\Sha_{div}(F,  G) =    \Sha_X(F, G).$$
   \end{theorem}

\begin{proof} Let  $P \in \XX$ be a closed point and $\zeta \in \Sha_{div}(F_P, G)$.
Let $\eta$ be a codimension one point of  Spec($ \hat{\OO}_{\XX, P}$).Then 
the discrete valuation on $F_P$ given by $\eta$ is a divsorial discrete valuation. 
Since $\zeta \in \Sha_{div}(F_P, G)$, $\zeta$ is trivial over $F_{P\eta}$. 

Suppose $\eta$ does not corresponds to the prime ideals $(\pi_P)$ and $(\delta_P)$.
Since $G$  is unramified on $\OO_{\XX, P}$ except possibly at $(\pi_P)$ 
and $(\delta_P)$,   $G$ is unramified at $\eta$. Since $\zeta$ is trivial over $F_{P\eta}$,  
 $\zeta$ is unramified at  $\eta$ (cf., \cite[Corollary A.6]{GP08}).  Hence $\zeta \in \Sha_{\pi_P\delta_P}(F_P, G)$.
 Therefore $\Sha_{div}(F_P,  G) = \Sha_{\pi_P\delta_P}(F_P, G)$. 
 
By  (\ref{exact-seq}), we have the following exact sequence 
$$ 1 \to \Sha_X(F, G) \to \Sha_{div}(F, G) \to \prod_{P \in X^{(1)}} \Sha_{\pi_P\delta_P}(F_P, G).$$
Since $G$ satisfies the local injectivity hypothesis (\ref{hyp1}),  $\Sha_{\pi_P\delta_P}((F_P, G)$ is trivial and hence 
$\Sha_X(F, G) = \Sha_{div}(F, G)$.
 \end{proof}

Let $X$ be the special fibre of the model $\XX$. Let $U$ be  an affine irreducible subset $U$ of $X$.
Let $R_U =  \cap _{P \in U} \OO_{\XX, P}$.  Then $T \subset R_U$. Let $t \in T$ be a parameter  and
 $\hat{R}_U$ be the $(t)$-adic completion of $R_U$ and $F_U$ the field of fractions of $\hat{R}_U$.  
 Let $P \in X$ be a closed point with $P$ in the closure $\bar{U}$ of $U$. Then 
 the generic point of $U$ gives a discrete valuation on $F_P$. Let $F_{U, P}$ be the completion of 
 $F_P$ at this discrete valuation.  Note that $F_{U,P}$ depends only on the generic point $\eta$ of $U$. We also denote 
 $F_{U,P}$ by $F_{P, \eta}$. 
 
 \begin{hyp}
\label{hyp3}   We say that 
$G$ satisfies the local    factorization  hypothesis  with respect to $\XX$   if 
for every nonsingular  closed point $P$ of $X$ and for every  affine irreducible subset $U$ of the special fibre  of $\XX$ with 
  $P$ in the closure of $U$, $G(F_{U, P}) = G(F_U)G(F_P)$. 
\end{hyp}

  \begin{prop} 
 \label{Up}
  Let $F$, $G$ and $\XX$  be as above.  
  Suppose that 
  $G$ satisfies the  local  factorization hypothesis (\ref{hyp3}).
Let $\PP_0$ be a finite set of closed points of $\XX$ containing all the singular points of $X$
and at least one closed point from each irreducible component of $X$. 
Let $V \subset X \setminus \PP_0$ be a nonempty open subset contained in an irreducible component of $X$.
Let $P \in V$ be a closed point and $U = V \setminus \{ P \}$. 
Then  the map 
$$H^1(F_V, G) \to H^1(F_U, G) \times H^1(F_P, G)$$
has trivial kernel.  
 \end{prop}
 
 \begin{proof} By (\cite[Theorem 2.4]{HHK3}) and (\cite[Proposition 3.9]{HHK5}), 
 the kernel of $H^1(F_V, G) \to H^1(F_U, G) \times H^1(F_P, G)$ is  in bijection with the double 
 cosets 
 $$G(F_U)  \,\backslash   G(F_{U, P})  
  \,/   G(F_{P}) .$$
  Since $G$ satisfies local    factorization   hypothesis  (\ref{hyp3}),  $G(F_{U, P}) = G(F_U)G(F_P)$. 
  Hence 
  $$H^1(F_V, G) \to H^1(F_U, G) \times H^1(F_P, G)$$
has trivial kernel. 
\end{proof}

Let $\PP$ be a finite set of closed points of $X$ containing all the singular points of $X$ and at least one point from 
 each irreducible component of $X$. Ley $\UU$ be the set of irreducible components of $X \setminus \PP$.
 Let 
$$\Sha_\PP(F, G) = ker (H^1(F, G) \to \prod_{U \in \UU} H^1(F_U, G) \times \prod_{P \in \PP} H^1(F_P, G)).$$

 \begin{theorem} 
 \label{pointsha}
  Let $F$, $G$ and $\XX$  be as above.  
  Suppose that 
  $G$ satisfies the  local factorization hypothesis (\ref{hyp3}).
Let $\PP_0$ be a finite set of closed points of $\XX$ containing all the singular points of $X$
and at least one closed point from each irreducible component of $X$. 
 Then  
$$ \Sha_X(F, G) = \Sha_{\PP_0}(F, G).$$ 
 \end{theorem}
 
 \begin{proof}  By (\cite[Corollary 5.9]{HHK3}), we have 
 $\Sha_X(F, G) = \cup_\PP \Sha_\PP(F, G)$, where $\PP$ runs over finite sets of closed points of $X$ containing 
 all the singular points of $X$ and at least one closed point from each irreducible component of $X$.
 Since $\Sha_{\PP}(F, G) \subseteq \Sha_{\PP'}(F, G)$ for all such finite sets of closed points of 
 $\PP$ and $\PP'$ with $\PP \subset \PP'$ (\cite[\S 5]{HHK3}), we have 
  $\Sha_X(F, G) = \cup_\PP \Sha_\PP(F, G)$, where $\PP$ runs over finite sets of closed points of $X$ containing $\PP_0$. 

    Let  $\PP$ be a finite set of closed points of $X$   containing 
 $\PP_0$ and   $\zeta \in \Sha_\PP(F, G)$.  Let $U$ be an irreducible component of $X \setminus \PP$
 and $V$  the  irreducible component of $X \setminus \PP_0$ containing  $U$.  Since 
 $\zeta \in \Sha_\PP(F, G)$, $\zeta$ is trivial over $U$ and trivial over all $P \in \PP$.  
 Suppose $U \neq V$. Let  $P \in V  \setminus U$. Since $\zeta$ is trivial over $F_P$,
 by (\ref{Up}), $\zeta$ is trivial over $F_{U \cup \{ P \}}$. Since $V \setminus U \subset \PP$ is a finite set, we get that 
 $\zeta$ is trivial over $F_V$. Hence $\zeta \in \Sha_{\PP_0}(F, G)$ and $\Sha_X(F, G) \subseteq \Sha_{\PP_0}(F, G)$.
Hence 
 $ \Sha_X(F, G) = \Sha_{\PP_0}(F, G).$
 \end{proof}

\begin{cor} 
 \label{dvrsha}
  Let $F$, $G$ and $\XX$  be as above.  
  Suppose that 
  $G$ satisfies local injectivity hypothesis (\ref{hyp1})
   and  local  factorization hypothesis (\ref{hyp3}).
Let $\PP_0$ be a finite set of closed points of $\XX$ containing all the singular points of $X$
and at least one closed point from each irreducible component of $X$. 
 Then  
$$ \Sha_{div}(F, G) = \Sha_{\PP_0}(F, G).$$ 
 \end{cor}

  \section{Simply connected strongly isotropic groups}
  \label{isotropic}

 For any algebraic group  $G$ over a field $L$, let $RG(L)$ denote the subgroup of 
 $G(L)$ consisting $R$-trivial elements and $G(L)/R = G(L)/RG(L)$  the set of $R$-equivalence classes (cf., \cite{Gil97}).
 
  Let $T$ be a complete discrete valuation ring with residue field $\kappa$ and field of 
fraction $K$. Let $F$ be the function field of a curve over $K$. 
  
  \begin{theorem} 
  \label{Sha-requiv}Let $F$ be as above and $G$ a semisimple simply connected   strongly isotropic  
 linear algebraic group  over $F$.  Let $\XX$ be a regular proper model of $F$
 and $\PP$ a finite set of closed points of $\XX$ containing all the singular points of the closed fibre  $X$ of 
 $\XX$ and at least one point from each irreducible  component of $X$.  Let $\UU$  be the set of irreducible 
 components of $X \setminus \PP$ and $\BB$ the set of branches. Then 
$$ \Sha_{\PP}(F, G)  \simeq \prod_{U \in \UU} G(F_{U}) /R \,\backslash \prod_{\wp \in \BB} G(F_\wp)/R 
  \,/ \prod_{P \in \PP} G(F_{P}) / R.$$ 
 \end{theorem} 
 
 \begin{proof} By (\cite[Corollary 3.6]{HHK3}), we have 
 $$ \Sha_{\PP}(F, G)  \simeq \prod_{U \in \UU} G(F_{U})  \,\backslash \prod_{\wp \in \BB} G(F_\wp)
  \,/ \prod_{P \in \PP} G(F_{P}) .$$ 
  Since $G$ is a semisimple simply connected   strongly isotropic  
 linear algebraic group  over $F$, by (\cite[Propsition 3.7]{GP}) we have 
 $$\prod_{\wp \in \BB} RG(F_\wp) =  \prod_{U \in \UU} RG(F_{U})  \prod_{P \in \PP} RG(F_{P}).$$
 Hence we have 
 $$ \Sha_{\PP}(F, G)  \simeq \prod_{U \in \UU} G(F_{U}) /R \,\backslash \prod_{\wp \in \BB} G(F_\wp)/R 
  \,/ \prod_{P \in \PP} G(F_{P}) / R.$$ 
 \end{proof}

Let $G$ be a connected linear algebraic group over $F$.
Suppose that  $\XX$  is a regular proper model of 
$F$  with   the union of the  ramification locus of $G$ and the closed fibre of $\XX$ 
is a union of regular curves with normal crossings. Let $P \in \XX$ be closed point.  
Then there exist $\pi_P, \delta_P \in  \OO_{\XX, P}$  such that the maximal ideal 
of $\OO_{\XX, P}$ is generated by $\pi_P$ and $\delta_P$, and 
$G$ is unramified on $\OO_{\XX, P}$ except possibly at $(\pi_P)$ and $(\delta_P)$.   We fix such a choice of 
$\pi_P$ and $\delta_P$ at each $P$. 

\begin{hyp}
\label{hyp2}   We say that 
$G$ satisfies the local surjectivity  hypothesis  with respect to $\XX$   if 
for every nonsingular  closed point $P$ of $X$ and for every codimension one point $\eta$ of $X$ with 
$P$ in the closure of $\eta$,  the map $G(F_P) \to G(F_{P, \eta})/R$ is surjective. 
\end{hyp}

 \begin{prop} 
 \label{Up-simply-conn}
  Let $F$, $G$ and $\XX$  be as above.  
  Suppose that $G$ is  a semisimple simply connected  strongly isotropic   linear algebraic group  over $F$  and 
  $G$ satisfies the  local surjectivity hypothesis  (\ref{hyp2}).
  Then $G$ satisfies local factorization hypothesis (\ref{hyp3}).  
 \end{prop}
 
 \begin{proof} 
 Since $G$ is simply connected and strongly isotropic, by (\cite[Propsition 3.7]{GP}), $RG(F_{U, P}) = RG(F_U)RG(F_P)$.
 Since $G$ satisfies local surjectivity hypothesis (\ref{hyp2}), 
 it follows that $G$ satisfies local factorization property. 
 \end{proof}

\begin{theorem} 
 \label{point-sha}
  Let $F$, $G$ and $\XX$  be as above.   
Suppose  $G$   is a semisimple simply connected  strongly isotropic   linear algebraic group  over $F$  and 
satisfies the local  surjectivity     hypothesis (\ref{hyp2}) with respect to $\XX$. 
Let $\PP_0$ be a finite set of closed points of $\XX$ containing all the singular points of $X$
and at least one closed point from each irreducible component of $X$. 
Then  $$\Sha_X(F, G) = \Sha_{\PP_0}(F, G).$$ 
 \end{theorem}
 
 \begin{proof}  Follows from (\ref{pointsha}) and (\ref{Up-simply-conn}). 
 \end{proof}

 \begin{cor} 
 \label{dvrsha-requiv}
Let $F$, $G$ and $\XX$  be as above.   
Suppose  $G$   is a semisimple simply connected   linear algebraic group  over $F$ which is strongly isotropic  and 
satisfies the local injectiviry hypothesis (\ref{hyp1}) and the 
 local  surjectivity     hypothesis (\ref{hyp2}) with respect to $\XX$. 
Let $\PP_0$ be a finite set of closed points of $\XX$ containing all the singular points of $X$
and at least one closed point from each irreducible component of $X$. 
Then 
 $$\Sha_{div}(F,  G)  \simeq \prod_{U \in \UU_0} G(F_{U}) /R \,\backslash \prod_{\wp \in \BB_0} G(F_\wp)/R 
  \,/ \prod_{P \in \PP_0} G(F_{P}) / R.$$ 
where  $\UU_0$ is the set of irreducible components of $X \setminus \PP_0$ and 
$\BB_0$ is the set of branches with respect to $\PP_0$. 
 \end{cor}
 \begin{proof}  Follows from (\ref{dvrsha}) and (\ref{Sha-requiv}). 
 \end{proof}

 \section{Field extension over two dimensional complete fields} 
\label{field-extns}

Let $R$ be a   two dimensional  complete  regular local  ring with  maximal ideal $ m = (\pi, \delta)$, residue field 
 $\kappa$ and  field of fractions $F$.  In this section we describe finite extensions of $F$ of degree coprime to char$(\kappa)$
 which are unramified on $R$ except possibly at $(\pi)$ and $(\delta)$.

 For any prime $\theta \in R$, let $F_\theta$ denote the completion of 
 $F$ at the discrete valuation given by $\theta$.  Note that the residue field $\kappa(\pi)$  of $F_\pi$  is 
  the field of fractions of 
 $R/(\delta)$ and hence a  complete discretely valued field with the image $\bar{\delta}$ of $\delta$ as a parameter. 
 
  Let $E/F$ be  a finite  extension of degree $n$. Suppose that 
 $n$ is coprime to char$(k)$.  Then $E/F$ is separable. 
 Let $\widetilde{R}$ be the integral closure of $R$ in $E$.

\begin{prop}
\label{unramified} If $E/F$ is unramified on $R$, then $E\otimes F_\pi$   is a field. 
\end{prop}

\begin{proof}   
Since $E/F$ is unramified $R$, 
$\widetilde{R}$ is a regular local ring with maximal ideal $(\pi, \delta)\widetilde{R}$. In particular  $\pi$ is a
prime in $\widetilde{R}$ and hence $\pi$ is inert in $\widetilde{R}$. Thus $E\otimes F_\pi \simeq E_\pi$ is a field. 
\end{proof} 

\begin{prop} 
\label{pi-galois}  If $E/F$  is Galois and   unramified on $R$ except possibly at $\pi$, 
then $E = E_1(\sqrt[e]{u\pi})$ for some
 extension $E_1/F$ which is unramified on $R$ and $u \in E_1$ a unit in the integral closure of 
 $R$ in $E_1$.  
 \end{prop}

\begin{proof} 
Let $Q$ be a prime ideal of $\widetilde{R}$ lying over the prime ideal $(\pi)$ of $R$.
Let $H = \{ \sigma \in G(E/F) \mid  \sigma(Q) = Q\} $ be the decomposition group of $Q$.
Let $N = E^H \subseteq E$. Then  $N/F$ is unramified at $(\pi)$  and 
$N\otimes F_\pi $ is isomorphic to the product of $F_\pi$. 
 Since $E/F$ is unramified on $R$ except possibly at 
$(\pi)$, $N/F$ is unramified on $R$.  Hence, by (\ref{unramified}), $N\otimes F_\pi$ is a field.
In particular $N = F$, $H = G(E/F)$ and $Q$ is the unique prime ideal of $\widetilde{R}$ lying 
over $(\pi)$. Thus $E\otimes F_\pi$ is a field. 

Let $I = \{ \sigma \in G(E/F) \mid \sigma (x) = x {\rm ~mod ~} Q {\rm ~for~ all~} x \in \widetilde{R} \}$
be the inertia group and $E_1 = E^I$. Then $E_1/F$ is unramified at $(\pi)$.
Hence $E_1/F$ is unramified on $R$.  Then  the integral closure $\widetilde{R}_1$ of $R$ in $E_1$ is  a regular 
local ring with maximal ideal generated by $(\pi, \delta)$. 

Since   
$[E: F]$ is coprime to char$(k)$,  $[E\otimes_{E_1} E_{1\pi} : E_{1\pi}] = [E : E_1] = e$ is coprime to char$(k)$.
Since   
$E\otimes_{E_1}  E_{1\pi} /E_{1\pi}$ is totally ramified, 
$ E\otimes_{E_1}  E_{1\pi} = E_{1\pi}(\sqrt[e]{u\pi})$ for some $u \in E_1$ which is a 
unit in $\widetilde{R}_1$.  Since $E\otimes_{E_1}  E_{1\pi} /E_{1\pi}$ is Galois,
$E\otimes_{E_1}  E_{1\pi} /E_{1\pi}$ is cyclic and $E_{1\pi}$ contains a primitive 
$e^{th}$ root of unity.  Since $\widetilde{R}_1$ is a complete regular local ring, it follows that 
$E_1$ contains a primitive 
$e^{th}$ root of unity. 

Since 
$E/E_1$  Galois and $G(E/E_1) = G(E\otimes_{E_1}  E_{1\pi} /E_{1\pi})$,  
it follows that $E/E_1$ is cyclic. 
 Thus $E = E_1(\sqrt[e]{a})$ for some $a \in E_1$.   Since $E_1$ is a the field of fractions of 
$\widetilde{R}_1$, we assume that $ a\in \widetilde{R}_1$. Since $\widetilde{R}_1$ is a UFD, we have 
$a = u \pi^r \delta^s \theta_1^{d_1} \cdots \theta_m^{d_m}$ for some primes 
$\theta_1, \cdots , \theta_m \in \widetilde{R}_1$ and unit $u \in \widetilde{R}_1$. 
Since $E/E_1$ is unramified on $\widetilde{R}_1$ except at $(\pi)$, $e$ divides $s$, $d_1$, $ \cdots$, $d_m$.
Hence $E = E_1(\sqrt[e]{u\pi^r})$. Since $E\otimes _{E_1}  E_{1\pi} /E_{1\pi}$ is totally ramified, 
$r$ is coprime to $n$ and hence $E = E_1(\sqrt[e]{v\pi})$ for some unit $v \in \widetilde{R}_1$. 
\end{proof}

\begin{cor}
\label{pi}
 If $E/F$ is  unramified on $R$ except possibly at $(\pi)$, then 
  $E\otimes F_\pi$  and $E \otimes F_\delta$ are   fields.
\end{cor}

\begin{proof} Let $L/E$ be the normal closure of $E/F$. 
Since $L/F$ is separable, $L/F$ is a Galois extension. 
Since $E/F$ is unramified on $R$ except possibly at $\pi$, $L/F$ is also 
unramified on $R$ except possibly at $\pi$.  Hence replacing $E$ by $L$, we assume that 
$E/F$ is a Galois extension.  Thus, by (\ref{pi-galois}), we have 
$E = E_1(\sqrt[m]{u\pi})$ for some
 extension $E_1/F$ which is unramified on $R$ and $u \in E_1$ a unit in the integral closure $\widetilde{R}_1$  of 
 $R$ in $E_1$.  Now it is easy to see that $E\otimes F_\pi$ and $E\otimes F_\delta$ are fields. 
\end{proof}

\begin{prop} 
\label{pi-delta} If $E/F$ is  unramified on $R$ except possibly at $(\pi)$ and $(\delta)$, then 
  $E\otimes F_\pi$  and $E \otimes F_\delta$ are   fields.
 \end{prop}

\begin{proof}  Let $Q$ be a prime ideal in $\widetilde{R}$  lying over $(\pi)$.
Let $H = \{ \sigma \in G(E/F) \mid  \sigma(Q) = Q\} $ be the decomposition group of $Q$.
Let $N = E^H \subseteq E$. Then  $N/F$ is unramified at $(\pi)$  and 
$N\otimes F_\pi $ is isomorphic to the product of $F_\pi$. 
 Since $E/F$ is unramified on $R$ except possibly at 
$(\pi)$ and $(\delta)$, $N/F$ is unramified on $R$ except possibly at $(\delta)$. 
 Hence, by (\ref{pi}), $N\otimes F_\pi$ is a field.
In particular $N = F$, $H = G(E/F)$ and $Q$ is the unique prime ideal of $\widetilde{R}$ lying 
over $(\pi)$. Thus $E\otimes F_\pi$ is a field.  Similarly $E\otimes F_\delta$ is a field. 
 \end{proof}

 \begin{prop} 
\label{pi-delta-galois} If $E/F$ is  Galois and  unramified on $R$ except possibly at $(\pi)$ and $(\delta)$, 
then 
  $E = E_1(\sqrt[e_1]{u\pi})(\sqrt[e_2]{v(\sqrt[e_1]{u\pi})^r\delta})$ for
  some extension $E_1/F$ which is unramified on $R$, $u \in E_1$ a unit in 
  the integral closure of $R$ in $E_1$ and $v \in E_1(\sqrt[e_1]{u\delta})$ a 
  unit in the integral closure of $R$ in $E_1(\sqrt[e_1]{u\delta})$. 
 \end{prop}

\begin{proof}  Since $E/F$ is unramified on $R$ except possibly at $(\pi)$ and $(\delta)$.
by (\ref{pi-delta}), $E \otimes F_\delta$ is a field.  Hence  there is only one 
prime $Q$ lying over the prime ideal $(\delta)$. In particular  the decomposition  group of 
$Q$ over $(\pi)$ is equal to the $G(E/F)$ and the inertia group $I$ of $Q$ is a normal subgroup of 
$G(E/F)$. Let $L = E^I$. Then $L/F$ is unramified at $(\delta)$ and hence unramified 
on $R$ except possibly at  $(\pi)$.  Thus, by (\ref{pi}), we have 
$L = E_1(\sqrt[e_1]{u\pi})$ for some $E_1/F$ unramified on $R$ and a 
 unit $u$ in the integral closure $\widetilde{R}_1$ of $R$ in $E_1$. 
 
 Since $E_1/F$ is unramified on $R$, $\widetilde{R}_1$ is a regular local ring of with maximal ideal 
 $(\sqrt[e_1]{u\pi}, \delta)$. Further $E/E_1$ is totally ramified  at $(\delta)$
 and hence $E\otimes E_{1\delta} = E_{1\delta}(\sqrt[e_2]{v\delta})$ for some $v \in E_{1\delta}$
 a unit in the valuation ring. In particular  $E \otimes E_{1\delta}/E_{1\delta}$ is a cyclic extension. 
As in the proof of (\ref{pi-galois}), we have $E = E_1(\sqrt[e_1]{u\pi})(\sqrt[e_2]{v(\sqrt[e_1]{u\pi})^r\delta})$.  
\end{proof}
 
\begin{lemma} 
\label{field-extension-branch}
Let $E_\pi/F_\pi$ be a finite extension of degree    coprime to char$(\kappa)$. 
Then  $E_\pi = (E_0 \otimes_F F_\pi)(\sqrt[e_2]{v\delta})(\sqrt[e_1]{u(\sqrt[e_2]{v\delta})^t\pi})$
for some extension  $E_0/F$ which is unramified  on $R$, $u, v $ units in the integral closure of 
$R$ in $E_0$  and $e_1, e_2, t \geq 0$. 
\end{lemma}

\begin{proof}    
Let $E_\pi(\pi)$ be the residue field of $E_\pi$. Then $E_\pi(\pi)$ is a finite extension of $\kappa(\pi)$ and hence
a complete discretely valued field. 
Let  $E_\pi(\pi)(\delta)$ be the residue field of 
$E_\pi(\pi)$.  Since  $[E_\pi : F_\pi]$ is coprime to char$(\kappa)$,  $E_\pi(\pi)(\delta)$ is a finite separable  extension of $\kappa$. Write 
$E_\pi(\pi)(\delta) \simeq \kappa[x]/(g(x))$ for some monic polynomial $g(x) \in \kappa[x]$.
Let $g(x) \in R[x]$ be a monic polynomial with the image in $\kappa[x]$ equal to $g(x)$.
Let $E_0 = F[x]/(f(x))$.  Then $E_0/F$   is unramified on $R$.
 Let $S_0$ be the integral closure of $R$ in $E_0$.
 Then $S_0$ is a regular local ring with  maximal ideal $(\pi, \delta)$ and the residue field  $E(\pi)(\delta)$. 
 
 Let $E_0(\pi)$ be the field of fractions of $S_0/(\pi)$. Then $E_0(\pi)$ is the maximal unramified extension of 
 $E_\pi(\pi)$ and $E_\pi(\pi)/E_0(\pi)$ a totally ramified extension. Since $[E_\pi : F_\pi]$ is coprime to char$(\kappa)$,  
 $E_\pi(\pi) = E_0(\pi) (\sqrt[e_2]{ \bar{v} \bar{\delta}})$ for some $v \in S$ a unit. 
 
 Let $E_1 = E_0(\sqrt[e_2]{v\delta})$. Then $E_1\otimes_F F_\pi$ is the maximal unramified extension of $E_\pi/F_\pi$.
 Hence $E_\pi =  (E_1\otimes F_\pi)(\sqrt[e_1]{a \pi})$ for some $a \in E_1\otimes F_\pi$ which is a unit in the valuation ring of 
 $E_1\otimes F_\pi$.   Since   
  residue field of $E_1\otimes F_\pi$ is isomorphic to 
 $E_\pi(\pi) = E_0(\pi) (\sqrt[e_2]{ \bar{v} \bar{\delta}})$, $\bar{a} = b (\sqrt[e_2]{\bar{v}})^t c^n$ for some 
 $c \in  E_\pi(\pi)$  and $b \in E_0(\pi)$ a unit in the valuation ring. 
 Since $E_0(\pi)$ is a the field of fractions of $S_0/(\pi)$, we have $b = \bar{u}$ for some $u \in S_0$ a unit. 
 Hence we have  
 $$E_\pi =  (E_1\otimes F_\pi)(\sqrt[e_1]{u (\sqrt[e_2]{v})^t \pi}) = 
 (E_0 \otimes_F F_\pi)(\sqrt[e_2]{v\delta})(\sqrt[e_1]{u(\sqrt[e_2]{v\delta})^t\pi})$$
 for some units $u,v \in S_0$. 
\end{proof}

  The following is  proved in  (\cite[Lemma 5.1]{PPS}) for Galois extensions. 
 \begin{lemma} 
\label{field-extension}
Let $E_\pi/F_\pi$ be a finite extension of degree $n$  coprime to char$(\kappa)$. 
 Then there exists a field extension $E/F$ of degree   $n$   which is unramified on $R$ except possibly at 
 $(\pi)$ and $(\delta)$ such that $E\otimes F_\pi \simeq E_\pi$.  
 Further \\
 i) if   $E_\pi/F_\pi$ is unramified,   then $E/F$ is  unramified on $R$ except possibly at $(\delta)$ \\
 ii)  if   $E_\pi/F_\pi$ is unramified and the residue field $E_\pi(\pi)$ of $E_\pi$ is unramified over $\kappa(\pi)$, then 
 $E/F$ is unramified on $R$  \\
 iii)  if $E_\pi/F_\pi$ is Galois, then $E/F$ is Galois. 
\end{lemma}

\begin{proof}   By (\ref{field-extension-branch}), there exists an extension 
$E_0/F$ which is unramified on $R$ such that 
$E_\pi = (E_0 \otimes_F F_\pi)(\sqrt[e_2]{v\delta})(\sqrt[e_1]{u(\sqrt[e_2]{v\delta})^t\pi})$
for some  $u, v $ units in the integral closure $S_0$ of $R$ in $E_0$  and $e_1, e_2, t \geq 0$.
Let $E = E_0((\sqrt[e_2]{v\delta})(\sqrt[e_1]{u(\sqrt[e_2]{v\delta})^t\pi})$.
Then $E\otimes F_\pi \simeq E_\pi$.  Clearly $E/F$ is unramified on $R$ except possible at 
$(\pi)$ and $(\delta)$. 

Suppose $E_\pi/F_\pi$ is unramified.  Then $e_1 = 1$ and 
 $E= E_0(\sqrt[e_2]{v\delta})$. Hence   $E/F$ is unramified on $R$ except possibly at $(\delta)$.
 
 Suppose that $E_\pi/F_\pi$ is unramified and the residue field $E_\pi(\pi)$ of $E_\pi$ is unramified over $\kappa(\pi)$.
 Then $e_1 = e_2 = 1$. Hence $E = E_0$ is unramified on $R$. 

Suppose that $E_\pi/F_\pi$ is Galois. Then $E/F$ is Galois follows as in the proof of 
(\cite[Lemma 5.1]{PPS}). 

\end{proof}

 \begin{lemma} 
\label{regular}
Let   $E_0/F$  be an extension  which is unramified over $R$.
Let $e \geq 1$ be  coprime to char$(\kappa)$. Let $E = E_0(\sqrt[e]{v\delta})$ for some unit $v$ in the integral closure of 
$R$ in $E_0$.
Then  the    integral closure $S$  of $R$ in $E$ is a complete regular local ring of dimension 2
 and    the maximal ideal of $S$  is given by  $(\pi,  \sqrt[e]{v\delta})$.
 \end{lemma}

\begin{proof}   Let $S_0$ be the integral closure of $R$ in $E_0$. Since $E_0/F$ is unramified on $R$,
$S_0$ is a regular local ring with maximal ideal $(\pi, \delta)$.
Since $e$ is coprime to char$(\kappa)$, $S$ is a regular local ring with 
maximal ideal $(\pi,  \sqrt[e]{v\delta})$ (\cite[Lemma 3.2]{PS2014}).
\end{proof}

The following is proved in (\cite{Sumit}) for Galois extensions. 
 \begin{lemma} 
\label{norm}
Let  $E_0/F$ be a field extension which is  unramified  on $R$. Let $u, v \in E$ be  units  in the integral closure of 
$R$ in $E_0$. Let   $E = E_0(\sqrt[e_2]{v\delta})(\sqrt[e_1]{u (\sqrt[e_2]{v\delta})^t \pi})$ for some $e_1, e_2, t \geq 0$.
Let $\lambda = w\pi^r\delta^s \in F$ for some  unit $w \in R$  and $r, s \in \Z$. 
If $\lambda$ is a   norm from $E \otimes F_\pi/F_\pi$, then $\lambda$ is a   norm from $E/F$.  
\end{lemma}

\begin{proof}  Suppose that  $\lambda$ is a   norm from $E \otimes F_\pi/F_\pi$.
Without loss of generality we assume that 
$0 \leq r, s < n =  [E : F]$. 

Let $z \in E \otimes F_\pi/F_\pi$ such that $N_{E \otimes F_\pi/F_\pi}(z) = \lambda.$
Let $E_1 = E_0(\sqrt[e_2]{v\delta})$.  Let $\pi' = \sqrt[e_1]{u (\sqrt[e_2]{v\delta})^t \pi} \in E$. 
Since $E_1 \otimes F_\pi$ is the maximal unramified extension of $E \otimes F_\pi/F_\pi$
and $ \pi'$ is a parameter in $E\otimes F_\pi$,
we have $z = a   \pi^{'m} b^n$ for som  unit $a$  in the valuation ring of 
$E_1 \otimes F_\pi$,  $b \in E\otimes F_\pi$ and $m \in \Z$. 
Since $\pi' \in E$ and $N_{E/F}(\pi') = w' \pi^{r'} \delta^{s'}$ for some $w' \in R$ a unit and $r', s' 
\in \Z$, replacing $\lambda$ by $\lambda N_{E/F}(\pi')^{-m}$, we assume that 
$\lambda = N_{E\otimes F_\pi/F_\pi}(a b^n)$ for some $a \in E_1  \otimes F_\pi$ 
a unit in the valuation ring of 
$E_1 \otimes F_\pi$ and $b \in E\otimes F_\pi$. 

Let $E_1(\pi)$ be the residue field of $E_1 \otimes F_\pi$ and $\bar{a} \in E_1(\pi)$ the image of  $a$.
Since $a \in E_1\otimes F_\pi$ is a unit in the valuation ring of $E_1\otimes F_\pi$, $\bar{a} \neq 0$. 
Let $E_0(\pi)$ be the residue field of $E_0 \otimes F_\pi$. Then we have 
$E_1(\pi) = E_0(\pi)(\sqrt[e_2]{\bar{v}\bar{\delta}})$ and $E_0(\pi)$ is the maximal unramified extension of 
$E_1(\pi)$. Hence we have $\bar{a} = \bar{a}_1 (\sqrt[e_2]{\bar{v}\bar{\delta}})^{m'} \bar{b_1}^n$
for some $a_1 \in E_0\otimes F_\pi$   and $b_1 \in  W_1$  which are   units  in the valuation rings. 
Further $\bar{a}_1 \in E_0(\pi)$ is a unit in the valuation ring of  $E_0(\pi)$. 

Let  $S_0$ be the integral closure of $R$ in $S_0$. 
Since $E_0/F$ is unramified in $R$, $S_0$ is a regular local ring with maximal ideal $(\pi, \delta)$.
Since $E_0(\pi)$ is the field of fractions of $S_0/(\pi)$ and $\bar{a}_1 \in E_0(\pi)$ is a unit in the valuation ring of  $E_0(\pi)$,
there exists a unit $x \in E_0$ such that $\bar{x} = \bar{a}_1 \in S_0/(\pi)$. 

Let $z_1 = x  (\sqrt[e_2]{v \delta})^{m'} $.  Then it is easy to see that $z_1^{-1} z \in (E\otimes F_\pi)^n$.
Since  $x \in S_0$ is a unit,  $N_{E/F}(z_1)$ is equal to a unit times a power of $\delta$. Hence 
replacing $\lambda$ by $\lambda N_{E/F}(z_1)^{-1}$,we assume that $\lambda = N_{E\otimes F_\pi/F_\pi}(z)$ with 
$z$ close to 1.  In particular $z \in (E\otimes F_\pi)^{*n}$ and hence 
$\lambda = N_{E\otimes F_\pi/F_\pi}(z) \in  F_\pi^{*n}$. 
Since $\lambda = w\pi^r \delta^s$, $r, s $ are divisible by $n$. 
Hence $w \in F_\pi^{*n}$. Since $w \in R$ is a unit, it follows that $w \in F^{*n}$.
Thus $\lambda \in N_{E/F}(E^*)$. 
\end{proof}

 \section{ Central simple algebras and reduced norms over two dimensional complete fields } 
\label{csa-2dim}

 Let $R$ be a   two dimensional  complete  regular local  ring with  maximal ideal $ m = (\pi, \delta)$, residue field 
 $\kappa$ and  field of fractions $F$. For any prime $\theta$, let $F_\theta$ be the completion of $F$ at 
 the discrete valuation given by $\theta$.  If  $\theta$ is a regular prime,  then $R/(\theta)$ is a complete discrete valuation ring 
 and hence    $\kappa(\theta)$ is   a complete discretely valued field. 
 Let $E/F$ be an extension which is unramified on $R$. 
 Let $S$ be the integral closure of $R$ in $E$. Let $\theta$ be a regular prime in $R$.
 Then $\theta$ is a prime in $S$ and we denote the residue field  of $E$ at $\theta$ by $E(\theta)$.

 \begin{lemma} 
\label{one}
Let  $D$ be a central simple  algebra over $F$ which is unramified on $R$ except possibly at $(\pi)$. 
 Suppose that the period of $D$ is coprime to char$(\kappa)$. 
  Then there exists a central simple algebra $D_0$ over $F$ which is unramified on $R$ and a cyclic extension 
$E/F$ which is unramified on $R$ such that $$D= D_0  +  (E,\sigma, \pi)$$ for some generator $\sigma$ of 
Gal$(E/F)$. 
\end{lemma}

\begin{proof}  Let $(E_0, \sigma_0)$ be the residue of $D$ at $\pi$. 
Let $E_\pi/F_\pi$ be the unique unramified extension with residue field $E_0$.
Since $E_0/\kappa(\pi)$ is cyclic, $E/F_\pi$ is cyclic.  
Since $D$ is unramified on $R$ except possibly at $(\pi)$,
the extension $E_0/\kappa(\pi)$ is an unramified extension (cf. \cite[Lemma 5.2]{PPS}). 
Since $E_\pi/F$ is a cyclic extension,  by (\ref{field-extension}), 
there exists a cyclic extension $E/F$ which is unramified on $R$ such that $E\otimes F_\pi \simeq E_\pi$. 
Let $\sigma \in $ Gal$(E/F)$  be the lift of $\sigma_0$  and  $D_0 = D - (E,\sigma, \pi)$.
Then $D_0$ is unramified on $R$ and $D = D_0 + (E,\sigma, \pi)$.  
\end{proof}

\begin{lemma} 
\label{split}
Let  $D$ be a central simple  algebra over $F$ which is unramified on $R$ except possibly at $(\pi)$ and $(\delta)$. 
 Suppose that the period of $D$ is coprime to  char$(\kappa)$.   If $D\otimes F_\delta$ is split, then $D$ is split. 
\end{lemma}

\begin{proof}  Suppose $D\otimes F_\delta$ is split. 
Then $D$ is unramified at $(\delta)$. Hence $D$ is unramified on $R$ except possibly at $(\pi)$.
By (\ref{one}), $D = D_0 + (E,\sigma, \pi)$ for some $D_0$ and $E/F$ unramified on $R$. 
For any regular prime  $\theta \in R$, let  $D_0(\theta)$ be the specialization of $D_0$ at $\theta$. 
Since   $D_0$ and $E/F$ are unramified at $\delta$ and $D\otimes F_\delta$ splits,
$D(\delta) = D_0(\delta) + (E(\delta)/\kappa(\delta), \sigma_\delta, \overline{\pi})$ is split algebra over $\kappa(\delta)$. 
Since $\kappa(\delta)$ is a complete discretely  valued  field  with $\bar{\pi}$ is a parameter, it follows that 
$E(\delta) = \kappa(\delta)$ and $D_0(\delta)$ is split over $\kappa(\delta)$. 
Since $R$ is complete $D_0$ is split and $E = F$. Hence $D$ is split. 
\end{proof}

\begin{lemma} 
\label{division}
Let $R$ be a   two dimensional  complete  regular local  ring with  maximal ideal $(\pi, \delta)$ and  field of fractions $F$.
Let  $D$ be a central division   algebra over $F$ which is unramified on $R$ except possibly at $(\pi)$ and $(\delta)$. 
 Suppose that the period of $D$ is coprime  to char$(\kappa)$. Then  $D\otimes F_\pi$ is division. 
\end{lemma}

\begin{proof} Let $D_\pi$ be the division algebra over $F_\pi$ which is Brauer equivalent to 
$D\otimes F_\pi$. 
 Since $F_\pi$ is a complete discretely valued field and period of $D$ is coprime to char$(\kappa(\pi))$,
there exists an unramified extension $E_\pi/F_\pi$ such that $E_\pi \subseteq D_\pi \otimes F_\pi$ a maximal subfield
 (cf., \cite[Lemma 5.1]{JW}). 
By (\ref{field-extension}), there exists an extension $E/F$ which is unramified on $R$ except possibly at $(\delta)$ such that 
$E\otimes_F F_\pi \simeq E_\pi$. 
Let $S$ be the integral closure of $R$ in $E$. Since $E/F$ is unramified on $R$ except possibly at $(\delta)$,
by the construction of $E/F$, 
$S$ is a complete two dimensional regular local ring with maximal ideal $m_S = (\pi, \delta')$ for some prime $\delta'$
in $S$ which lies over $\delta$ (cf. \cite[Lemma 3.1 and 3.2]{PS2014}).  Since $E\otimes_F F_\pi \simeq E_\pi$, $D\otimes_F  \otimes E_\pi$ is split, 
$D \otimes E$ is unramified on $S$ except possibly at $(\delta')$. 
Hence, by (\ref{split}), $D\otimes E$ is split. Hence ind$(D) \leq [E : F]$.
Since $[E : F] = [E_\pi, F_\pi] = $ ind$(D\otimes F\pi) \leq$ ind$(D) \leq [E :F]$, we have 
ind$(D\otimes F\pi) =$  ind$(D)$. 
\end{proof}

\begin{cor} 
\label{index}
Let  $D$ be a central simple   algebra over $F$ which is unramified on $R$ except possibly at $(\pi)$ and $(\delta)$. 
 Suppose that the period of $D$ is coprime to  to char$(\kappa)$.  Then  ind$(D) = $ ind$(D\otimes F_\pi)$. 
\end{cor}

\begin{lemma} 
\label{two}
Let  $D$ be a central simple  algebra over $F$ which is unramified on $R$ except possibly at $(\pi)$ and $(\delta)$. 
 Suppose that the period of $D$ is coprime to  to char$(\kappa)$. 
   Then there exists a central simple algebra $D_0$ over $F$ which is unramified on $R$, a cyclic extension 
$E_1/F$ which is unramified on $R$ and a cyclic extension 
$E_2/F$ which is unramified on $R$ except possibly at $(\pi)$  
 such that $$D= D_0  +  (E_1/F,\sigma_1, \pi) + (E_2/F, \sigma_2, \delta)$$ for some generators $\sigma_i$ of 
Gal$(E_i/F)$. 
\end{lemma}

\begin{proof}   
Let  $(E_\delta, \sigma_0)$ be the lift of the residue of   $D$ at $\delta$. 
Since $E_\delta /F_\delta$ is a cyclic extension which is unramified, there exists a cyclic extension $E_2/F$ 
which is unramified on $R$ except possibly 
at $(\pi)$ such that $E_2 \otimes F_\delta \simeq E_\delta$  (\ref{field-extension}).  
Let $\sigma_2 \in $ Gal$(E_2/F)$  be the lift of $\sigma_0$  and  $D_1  = D - (E_2/F,\sigma_2, \delta)$.
Then $D_1$ is unramified on $R$ except possibly at $\pi$.
Hence, by (\ref{one}),  
there exists a central simple algebra $D_0$ over $F$ which is unramified on $R$ and a cyclic extension 
$E_1/F$ which is unramified on $R$ such that $D_1= D_0  +  (E_1/F,\sigma_1, \pi)$.
Hence $D = D_0 + (E_1/F,\sigma_1, \pi) + (E_2/F,\sigma_2, \delta)$ as required. 
\end{proof}

\begin{cor} 
\label{index-cal}
Let  $D$ be a central simple   algebra over $F$ which is unramified on $R$ except possibly at $(\pi)$ and $(\delta)$. 
 Suppose that the period of $D$ is coprime to  to char$(\kappa)$. 
  Let $D = D_0 + (E_1/F,\sigma_1, \pi) + (E_2/F,\sigma_2, \delta)$ be as in (\ref{two}). 
Then 
$$ind(D) =   ind(D_0 \otimes E_1E_2) [E_1E_2 : F].$$ 
\end{cor}

\begin{lemma} 
\label{splitting-field}
Let  $D$ be a central simple  algebra over $F$ which is unramified on $R$ except possibly at $(\pi)$ and $(\delta)$. 
 Suppose that the period of $D$ is coprime to  to char$(\kappa)$. 
  Let $E_\pi/F_\pi$ be a finite   extension of fields of degree coprime to char$(\kappa)$. If   $  D\otimes E_\pi$ is split, then there exists an extension 
$E/F$ of fields such that $E/F$ is unramified on $R$ except possibly at $(\pi)$ and $(\delta)$ such that 
$E\otimes F_\pi \simeq E_\pi$ and $D\otimes E$ is split. 
\end{lemma}

\begin{proof}   Let $E_{\pi, nr}$ be the maximal unramified extension of 
$E_\pi/F_\pi$. 
By (\ref{field-extension}),     
there exists an extension $E_1/F$ of fields such that $E_1/F$ is unramified on $R$ except possibly at   $(\delta)$ 
such that  $E_1\otimes F_\pi \simeq E_{\pi, nr}$.   Let $S_1$ be the integral closure  of $R$ in $E_1$. 
By (\ref{regular}), $S_1$ is a  complete regular local ring of dimension 2
 and there is a unique (up to units) $\delta_1 \in S_1$ which lies over $\delta$
 with the maximal ideal of $S_1$ given by  $(\pi, \delta_1)$. 
 
 Since $[E_\pi : F_\pi]$ is coprime to char$(\kappa)$, we have
$E_\pi = E_{\pi, nr}(\sqrt[e]{v\pi})$ for some $v \in E_{\pi, nr}$ which is unit in the valuation ring of $E_{\pi, nr}$. 
We have $v = u \delta_1^r a^e$ for some unit  $u \in S_1$, $r, s \in {\mathbb Z}$ and $a \in E_1$ (cf., \cite[Remark 7.1]{PS2022}). 
Let $E = E_1(\sqrt[e]{u \delta_1^r \pi})$. 

Suppose that $D\otimes E_\pi$ is split. 
Let $D_1 = D\otimes_F E_1$. Since $D\otimes E_\pi$ is split, we have 
$D_1 \otimes_{E_1} E_\pi$ is split.  Hence by (\ref{cyclic}), we have $D_1\otimes E_{\pi, nr} = (L_\pi, \sigma, u\delta_1^r\pi)$
for some cyclic extension $L_\pi/E_{\pi, nr}$ which is unramified. 

By (\ref{field-extension}), there exists a cyclic extension $L/E_1$ such that 
$L\otimes E_{1, \pi} \simeq L_\pi$. Let $\sigma_1$ be the generator of Gal$(L/E_1)$ which is the restriction of $\sigma$.
Then $(D_1 \otimes (L/E_1, \sigma, u\delta_1^r\pi)) \otimes E_{1\pi}$ is split. 
Hence, by (\ref{split}), $D_1 \otimes (L/E_1, \sigma, u\delta_1^r\pi)$ is split over $E_1$.
Since $(L/E_1, \sigma, u\delta_1^r\pi)\otimes_{E_1}E$ is split, $D_1 \otimes E$ is split.
In particular $D\otimes E$ is split. 
\end{proof}

The following is proved in (\cite{PPS}) if $\kappa$ is a finite field.

 \begin{prop} 
\label{reduced-norm}
Let  $A$ be a central simple  algebra over $F$ which is unramified on $R$ except possibly at $(\pi)$ and $(\delta)$. 
 Suppose that the period of $A$ is coprime to char$(\kappa)$. 
   Let $\lambda = u\pi^r\delta^s \in R$ for some $u \in R$ a unit and $r, s \geq 0$. 
If $\lambda$ is a reduced norm from $A\otimes F_\pi$, then $\lambda$ is a reduced norm from $A$.  
\end{prop}

\begin{proof} Let $A = M_n(D)$ for some central division algebra $D$ over $F$.
Since for any field extension $L/F$, Nrd$(A^*) = $ Nrd$(D^*)$, replacing $A$ by $D$, 
we assume that $A = D$ is division. 
 Suppose $\lambda$ is a reduced norm from $D\otimes F_\pi$.
Since $D\otimes F_\pi$ is a division algebra (\ref{division}),  
there exists a maximal subfield $E_\pi$ of $D\otimes F_\pi$ such that $\lambda$ is a norm from the 
extension $E_\pi/F_\pi$. Then, by (\ref{splitting-field}), there exists an extension 
$E/F$ which is unramified on $R$ except possibly at $(\pi)$ and $(\delta)$
such that  $E\otimes F_\pi \simeq E_\pi$ and $D\otimes E$ is split.  Further from the proof of 
(\ref{splitting-field}), $E$ is an in (\ref{norm}). 
Since $\lambda$ is a norm from $E_\pi/F_\pi$, by (\ref{norm}), $\lambda$ is a norm from $E/F$. 
Hence $\lambda$ is a reduced norm from $D$.
\end{proof}

\section{Local injectivity for groups of type $A_n$}
\label{local-inj}

Let $R$ be a   two dimensional  complete  regular local  ring with  maximal ideal $ m = (\pi, \delta)$, residue field 
 $\kappa$ and  field of fractions $F$.

\begin{prop} 
\label{sl1-pidelta}
Let  $A$ be a central simple  algebra over $F$ which is unramified on $R$ except possibly at $(\pi)$ and $(\delta)$. 
 Suppose that the period of $A$ is coprime to char$(\kappa)$. 
 Then $H^1_{\pi\delta}(F, SL_1(A)) \to H^1(F_\pi, SL_1(A))$ is an isomorphism.  
 \end{prop}

\begin{proof}  Let $R_{\pi\delta} = R[\frac{1}{\pi\delta}] \subset F$.
Then $R_{\pi\delta}$ is a principal ideal domain with field of fractions $F$. 
Since $A$ is unramified on $R$ except possibly at $(\pi)$ and $(\delta)$,
there exists an Azumaya algebra $\AA$ over $R_{\pi\delta}$ such that $\AA \otimes F \simeq A$ (cf. \cite[Lemma 3.1]{LPS}). 

Since $H^1_{et}(R_{\pi\delta},  GL_1(\AA))$ is trivial, the   exact sequence 
$ 1 \to SL_1(\AA) \to GL_1(\AA) \to \G_m \to 1$ 
of algebraic groups over $R_{\pi\delta}$ gives an isomorphism 
$$H^1_{et}(R_{\pi\delta},  SL_1(\AA)) \simeq R_{\pi\delta}^*/ Nrd(\AA^*).$$

Since $H^1_{\pi\delta}(F, SL_1(A)) \simeq H^1_{et}(R_{\pi\delta}, SL_1(\AA)$ (\cite[Theroem 6.13]{CTS}), 
we have $H^1_{\pi\delta}(F, SL_1(A)) \simeq R_{\pi\delta}^*/ Nrd(\AA^*).$
Let $\zeta \in H^1_{\pi\delta}(F, SL_1(A))$ which maps to the trivial element in $H^1(F_\pi, SL_1(A))$.
Let $\lambda \in R_{\pi\delta}^*$ representing $\zeta$.  Since $\lambda$ is a unit in $R_{\pi\delta}^*$,
we have $\lambda = u \pi^r\delta^s$ for some unit $u \in R$ and $r, s \in \Z$.
Since $\zeta$ maps to the trivial element in $H^1(F_\pi, SL_1(A)) \simeq F_\pi^*/Nrd(A\otimes F_\pi)$,
$\lambda$ is a reduced norm from  $A\otimes F_\pi$. Hence, by (\ref{reduced-norm}),
$\lambda$ is a reduced norm from $A$. In particular  the map  $H^1_{\pi\delta}(F, SL_1(A)) \to H^1(F_\pi, SL_1(A))$
is injective.  

Let $n = $ ind$(A)$. Since $n$ is coprime to char$(\kappa)$,  $R_{\pi\delta}^*  \to F_\pi^*/F_\pi^{*n}$ is 
surjective (cf., \cite[Remark 7.1]{PS2022}). 
Since $F_\pi^{*n} \subseteq Nrd(A^*)$,  the map  $H^1_{\pi\delta}(F, SL_1(A)) \to H^1(F_\pi, SL_1(A))$
is surjective. 
\end{proof}

 \begin{prop} 
\label{sl1-hyp1}
Let  $A$ be a central simple  algebra over $F$ which is unramified on $R$ except possibly at $(\pi)$ and $(\delta)$. 
 Suppose that the period of $A$ is coprime to char$(\kappa)$. 
 Then $H^1(F, SL_1(A)) \to  \prod_\nu H^1(F_\nu, SL_1(A))$ is injective, where
 $\nu$ running over all height one prime ideals of $R$. 
  \end{prop}

\begin{proof}  Let $\zeta \in  ker(H^1(F, SL_1(A)) \to  \prod_\nu H^1(F_\nu, SL_1(A))$.
Let $\nu$ be a height one prime ideal of $R$ which is not equal to $(\pi)$ and $(\delta)$.
Since $A$ is unramified at $\nu$ and $\zeta$ maps to the trivial element in $H^1(F_\nu, SL_1(A))$,
$\zeta$ is unramified at $\nu$. Hence $\zeta \in H^1_{\pi\delta}(F, SL_1(A))$.
Since the image of $\zeta$ in $H^1(F_\pi, SL_1(A))$ is trivial, by (\ref{sl1-pidelta}), 
$\zeta$ is trivial.  
\end{proof}
 
  Let $L = F(\sqrt{d})$ be a quadratic extension  with  
   $ d = u$ or $\pi$ or $\delta$  for some unit  $u \in R^*$. 
    Let $S$ be the integral closure of $R$ in $L$. Then 
    $S$ is a regular local ring with maximal ideal $(\pi_1,  \delta_1)$, where 
    $\pi_1$ and $\delta_1$ are unique (up to units) primes in $S$ lying over $\pi$ and $\delta$
    respectively.

 Let   $A$ be  a central simple algebra  over $L$ with a $L/F$-involution $\tau$.
 Suppose that $2per(A)$ is coprime to char$(\kappa)$ and $(A, \tau)$ is unramified on $R$
 except possibly at $(\pi)$ and $(\delta)$. 

 The following is proved in (\cite[Lemma 8.1]{PS2022}) when $\kappa$ is a finite field. 

 \begin{lemma}
  \label{scaling}
     Let $a = w \pi_1^r \delta_1^s \in L$ for some unit $w \in S$ and $r, s \in \Z$.
      Suppose there exists $\theta_\pi \in F_{\pi}^*$ such tat $a \theta_\pi \in Nrd(A\otimes F_\pi)$.
      Then there exists $\theta\in F$  such that 
      $a \theta \in Nrd(A)$.
       \end{lemma} 
       
\begin{proof} Using (\ref{sl1-hyp1}), the proof is similar to the proof of  (\cite[Lemma 8.1]{PS2022}). 
\end{proof}

\begin{prop}
\label{su-hyp11} 
    Let $\zeta \in H^1(F, SU(A,\tau))$ which maps to 
 the trivial element in $H^1(F, U(A, \tau))$. 
 If $\zeta$ maps to the trivial element  in $H^1(F_{\nu}, SU(A,\tau))$ for all height one prime ideals  $\nu$
 of $R$,  then $\zeta$ is the trivial element.  
 \end{prop}
 
 \begin{proof}  Suppose that $d$ is a square in $F$. Then $L \simeq F \times F$, 
 $A = A_0 \times A^{op}$  for some central simple algebra $A_0$ over $F$,
  $\tau(x, y) = (y, x)$  and  $SU(A, \tau) = SL_1(A)$  (\cite[p.346]{KMRT}). 
 Hence the result follows from (\ref{sl1-hyp1}). 
 
 Suppose that $d$ is not a square in $F$. Then 
 $L$ is a quadratic field extension of $F$.

 Let $R_{\pi\delta} = R[\frac{1}{\pi\delta}]$. Since $R$ is a principal ideal domain, $R_{\pi\delta}$ is
 also a principal ideal domain. Since $d \in R_{\pi\delta}$ is a unit, $S_{\pi\delta} =  R_{\pi\delta}[\sqrt{d}]$ is an 
 \'etale extension and $S_{\pi\delta}$ is a principal ideal domain.

Since $(A,\tau)$ is unramified on $R$ except possibly at $(\pi)$ and $(\delta)$,
$(A,\tau)$ is unramified on $R_{\pi\delta}$. 
Hence there exists an Azumaya algebra $\AA$ over $S_{\pi\delta}$ and 
  such that $\AA \otimes_{S_{\pi\delta}} L  = A$ and $\tau(\AA) = \AA$ (\ref{purity}). 
 Since $S_{\pi\delta}/R_{\pi\delta}$ is \'etale, the groups $U(\AA, {\bf \tau})$ and $SU(\AA, {\bf \tau})$
 are reductive groups over $R_{\pi\delta}$. 

 Let $\zeta \in H^1(F, SU(A,\tau))$ which maps to 
 the trivial element in $H^1(F, U(A, \tau))$. 
Suppose that  $\zeta$ maps to the trivial element  $H^1(F_{\nu}, SU(A,\tau))$
 for all height one prime ideals  $\nu$  of $R$. 
Since $\zeta$ maps to the trivial element in 
$H^1(F_{\nu}, SU(A,\tau))$
 for all height one prime ideals  $\nu$  of $R$,   
there exists $\zeta_0 \in H^1_{et}(R_{\pi\delta}, SU( \AA, \tau))$ such 
that $\zeta_0$ maps to $\zeta$ (\cite[Proposition 6.8]{CTS}). 
Since the image of $\zeta$ in $H^1(F, U(A,\tau))$ is trivial, 
the image of $\zeta_0$ in $H^1(R_{\pi\delta}, U(\AA, \tau))$ is trivial (see \cite{Nis84} and \cite[Th\'eor\`eme~I.1.2.2]{Gil94}). 

Since  $R^1_{S_{\pi\delta}/R_{\pi\delta}}\G_m (R_{\pi\delta}) = \{ a \tau(a)^{-1}  \mid a \in S^*_{\pi\delta} \}$
(cf. \ref{pid}), the exact sequence 
$$U(\AA, \tau)(R_{\pi\delta})  \to  R^1_{S_{\pi\delta}/R_{\pi\delta}}\G_m (R_{\pi\delta})  \to 
H^1(R_{\pi\delta}, SU(\AA, \tau)) \to  H^1(R_{\pi\delta}, U(\AA, \tau_))$$ gives an 
element $a \in  S_{\pi\delta}^*$   such that   the element  $a\tau(a)^{-1} $ in $ R^1_{S_{\pi\delta}/R_{\pi\delta}}\G_m (R_{\pi\delta}) $
 maps to $\zeta_0$. 

Since $a \in S_{\pi\delta}^*$, we have $a = w\pi_1^r \delta_1^s$ for some unit $w \in S$ and $r,s \in\Z$. 
Since the image of  the class of $\zeta$ in $H^1(F_{\pi}, SU(A, \tau))$ is trivial,  
$a\tau(a)^{-1}$ is in the image of $U(A, \tau)(F_{\pi})$. 
Hence there exists $b \in L_{\pi}^* \cap  Nrd(A\otimes F_\pi)$ such that $a\tau(a)^{-1} = b\tau(b)^{-1}$.
Since $L_\pi^\tau = K_\pi$ and $\tau(a^{-1}b) = a^{-1}b$, there exists $\theta_\pi  \in K_\pi$ such that 
$b = \theta_\pi a$.    Hence, by (\ref{scaling}), there exists $\theta \in F$ such that 
$\theta a \in Nrd(A)$. Since $a\tau(a)^{-1} = \theta a  \tau( \theta a )^{-1}$, replacing 
$a$ by $\theta a$, we assume that   $a \in Nrd(A)^*$. 
In particular  $a\tau(a)^{-1}$ is in the image of $U(A, \tau)(F)$.  Since $a\tau(a)^{-1}$ maps to $\zeta$,  
 $\zeta$ is trivial in $H^1(F, SU(A, \tau))$. 
 \end{proof}

 \section{The group $SK_1(D)$ over a discrete valued fied}
 \label{sk1d-dvr}
 
 Let $R$ be an integral domain with a discrete valuation   $\nu$  (not necessarily complete), 
 $R_\nu$ the valuation ring at $\nu$,  $m_\nu$ the maximal ideal of $R_\nu$, 
  $F$ the field of fractions of $R$ and    $\kappa(\nu)$ the residue field at $\nu$.  
Let $n$ be an integer coprime to char$(\kappa(\nu))$. 
Let $\hat{F}$ be the completion of $F$ at $\nu$. 

 Let $D$ be a central  division  algebra over $F$.
 Suppose that $D\otimes_F \hat{F}$ is division. 
We recall the following from ( \cite[CH. 5]{reiner}). 
 Let $\Lambda = \{ x \in D \mid Nrd_D(x) \in R_\nu \}$. Then $\Lambda$  is the unique maximal $R_\nu$-order of $D$ 
 and $m_D = \{ x \in D \mid Nrd_D(x) \in  m_\nu \}$ is the unique maximal (2-sided) ideal of $\Lambda$.
 There exists $\pi_D \in \Lambda$ such that $m_D = \Lambda \pi_D$ and 
 every nonzero  element $z \in \Lambda$ is of the form $u \pi_D^n$ for some $u \in \Lambda$ a unit and 
 $n \in {\mathbb Z}$. 
 Further $\bar{D} = \Lambda/m_D$ is a division algebra  (\cite[Theorem 12.8]{reiner}).
 Let $Z(\bar{D} )$ be the center of $\bar{D}$. Suppose  the  index  of $D$ is coprime to char$(\kappa(\nu))$.
 Then $Z(\bar{D})/\kappa(\nu)$ is a cyclic extension (cf. \cite[Proposition 1.7]{JW}). 
 For any $z \in \Lambda$, let $\bar{z} \in \bar{D}$ be the image of $z$. 
 The   inner automorphism given by  $\pi_D$ induces an automorphism $\bar{\tau}$ of $\bar{D}$.
 Let $\tau$ be the restriction of $\bar{\tau}$ to   $Z(\bar{D})$.
 Then $\tau$ is a generator   of  Gal$(Z(\bar{D})/K)$.

Let $\pi \in F$ be a parameter.  Since index of $D$ is coprime to char$(\kappa(\nu))$, we have 
$D = D_0 + (E, \sigma, \pi)$ for some   central division algebra $D_0$ over $F$ unramified at $\nu$
and   cyclic extension $E/F$ unramified at $\nu$.  Also we  have   
$E \subseteq D$,  the residue field of $E$ at the extension of $\nu$ is $Z(\bar{D})$ and  $\bar{D} = \bar{D}_0 \otimes Z(\bar{D})$. 
 Further the inner automorphism given by $\pi_D$  restricted to 
 $E$ is a generator of Gal$(E/F)$ which is a lift of    $\tau$. 
   
 Let $D_1 = C_D(E)$. Then $D_1/E$ is unramified at the extension of $\nu$ to $E$.
  Let  $S_\nu$ be the 
 integral closure of $R_\nu$ in $E$ and $\Lambda_1 = \{ x \in D_1 \mid Nrd_{D_1}(x) \in S_\nu \}$.
 Then $\Lambda_1$ is  the maximal  $S_\nu$-order of $D_1$
 and  $\Lambda_1 \subseteq \Lambda$  with   $\bar{D}_1  = \bar{D}$. 
 Since the inner automorphism given by $\pi_D$  restricted to  $E$ is an automorphism of $E/F$, 
 $\pi_D D_1 \pi_D^{-1} = D_1$ and hence  $\pi_D \Lambda_1 \pi_D^{-1} = \Lambda_1$. 
  
Let 
$$Z(\bar{D})^{*1} = \{  \beta \in Z(\bar{D})^* \mid N_{Z(\bar{D})/\kappa(\nu)}(\beta) = 1\}$$
and 
$$ G = Z(\bar{D})^{*1} \cap Nrd_{\bar{D}}(\bar{D}^*).$$

 Let $z \in SL_1(D)$. Then $z \in \Lambda$ is a unit.  Let $\beta = Nrd_{\bar{D}}(\bar{z}) \in Z(\bar{D})^*$. 
  Since $\overline{Nrd_D(z)} =  N_{Z(\bar{D})/\kappa(\nu)}(Nrd_{\bar{D}}(\bar{z}))$, 
   $ N_{Z(\bar{D})/\kappa(\nu)}(\beta)= 1$ (\cite[Lemma 2]{Platonov1976}).
  Hence we have a homomorphism of groups 
  $$\eta : SL_1(D) \to G$$
  given by $z \mapsto Nrd_{\bar{D}}(\bar{z})$. 
  
   Let $$H = \{ \beta \in Z(\bar{D})^* \mid \beta = \theta \tau(\theta)^{-1} {\rm ~ for ~ some ~}  \theta \in 
  Nrd_{\bar{D}}(\bar{D}^*) \}  \subseteq G.$$
  Let $x, y \in D^*$. Then  $x = u_1\pi_D^{n_1}$ and $y = u_2\pi_D^{n_2}$
  for some $u_1, u_2 \in \Lambda$ units and $n_1, n_2 \in \Z$.  
  Then $xyx^{-1}y^{-1} = u_1\pi_D^{n_1}   u_2\pi_D^{n_2} \pi_D^{-n_1}  u^{-1}_1  \pi_D^{-n_2} u^{-1}_2 = 
  u_1 ( \pi_D^{n_1}   u_2\pi_D^{-1}) ( \pi_D^{n_2}  u^{-1}_1  \pi_D^{-n_2})  u^{-1}_2.    $
Since for any $z \in \Lambda$,  Nrd$_{\bar{D}}(  \overline{\pi_D z \pi_D{-1}}) = 
  \tau$( Nrd$_{\bar{D}}(  \bar{  z}))$,
  we have 
  $$ 
  \begin{array}{rcl}
  Nrd_{\bar{D}}(\overline{xyx^{-1}y^{-1}})  & = &  Nrd_{\bar{D}}(  \bar{u}_1) 
 Nrd_{\bar{D}}(   \overline{( \pi_D^{n_1}   u_2\pi_D^{-1})})   Nrd_{\bar{D}}(   \overline{ ( \pi_D^{n_2}  u^{-1}_1  \pi_D^{-n_2})})   Nrd_{\bar{D}}(   \overline{u^{-1}_2}) \\ 
 & = & Nrd_{\bar{D}}(  \bar{u}_1)  \tau ( Nrd_{\bar{D}}(   \bar{ u}_2   )
\tau(  Nrd_{\bar{D}}(   \bar{u}^{-1}_1   )  )  Nrd_{\bar{D}}(  \bar{u}^{-1}_2 ) \\
& = &  Nrd_{\bar{D}}(  \bar{u}_1 \bar{u}^{-1}_2)  \tau ( Nrd_{\bar{D}}( \bar{u}_2  \bar{u}^{-1}_1 ) ) \\
& = &  Nrd_{\bar{D}}(  \bar{u}_1 \bar{u}^{-1}_2) \tau ( Nrd_{\bar{D}}( \bar{u}_1  \bar{u}^{-1}_2 ) )^{-1} \in H.
 \end{array}
 $$
  Hence we have the induced homomorphism 
  $$ \bar{\eta} : SK_1(D) \to G/H.$$

 \begin{prop} 
 \label{complex} The sequence  
 $$
SK_1(D_1) \to SK_1(D)  \buildrel{\bar{\eta}} \over{\to}  G/H \to 1 $$
is a complex. 
\end{prop}

 \begin{proof}  
 
 Let $z \in SL_1(D_1)$.  Then Nrd$_{\bar{D}}(\bar{z}) =$ Nrd$_{\bar{D}_1}(\bar{z}) = 
 \overline{Nrd_{D_1}(z)} = 1$. Hence $\bar{\eta}(z) = 1$ and the image of $SK_1(D_1)$ is $SK_1(D)$ is in the kernel of 
 $\bar{\eta}$.  
  \end{proof}

  \begin{lemma}
  \label{eta-surj}Suppose that $F$ is complete. 
  Then the  homomorphism  $\eta$ is surjective. 
  \end{lemma}
  
  \begin{proof}    
 Let $\beta \in G$.
  Since $N_{Z(\bar{D})/\kappa(\nu)}(\beta) = 1$, 
 there exists $\theta \in Z(\bar{D})$ such that $\beta = \theta^{-1}\tau(\theta)$. 
 Let $a \in E^*$ be a lift of $\theta$ and $b  = a^{-1}\tilde{\tau}(a) \in E^*$. 
 Since  $\beta$ the image  of $b$ is a reduced norm from $\bar{D}$,  $D_1$ is unramified and 
 $F$ is complete,  $b$ is a reduced norm from $D_1$.
 Let $z \in D_1$ be such that Nrd$_{D_1}(z) = b$.  
 Since $Nrd_D(z) = N_{E/F}(Nrd_{D_1}(z)) = N_{E/F}(b) = N_{E/F}(a^{-1}\tilde{\tau}(a)) = 1$,
 $z \in SL_1(D)$.  Since $Nrd_{\bar{D}}(\bar{z}) = \beta$, $\eta$ is surjective. 
  \end{proof}

\begin{prop}(\cite[Corollary 3.12]{Platonov1976})
\label{1modt}  Suppose that $F$ is complete.
Let $x \in SL_1(D)$. If $\bar{x} = 1  \in \bar{D}$, then 
$x \in [D^*, D^*]$
\end{prop}

\begin{proof} Suppose $\bar{x} = 1 \in \bar{D}$.  
Let $L$ be a maximal subfield of $D$ containing $x$.  Then $[L : F] = n$. 
Let $e$ be the ramification index of $L/F$ and $L_0$ the residue field $L$. 
Since $n$ is coprime to char$(\kappa(\nu))$ and $F$ is complete, 
there exists $y \in  L$ such that $\bar{y} = 1 \in L_0$ and $y^n = x$. 
Then Nrd$_D(x) = N_{L/F}(x)  = N_{L/F}(y)^n= 1$.
  Since $ \overline{N_{L/F}(y)}  = N_{L_0/\kappa(\nu)}(\bar{y})^e$   and $\bar{y} = 1$,
  we have $ \overline{N_{L/F}(y)} = 1$. Since $n$ is coprime to char$(\kappa(\nu))$,
  $N_{L/F}(y)^n= 1$ and $ \overline{N_{L/F}(y)} = 1$, it follows that 
  $  N_{L/F}(y) = 1$. In particular $y \in SL_1(D)$.  
  Since $x = y^n$,   $x\in [D^*, D^*]$ (\cite[Lemma 2, p. 157]{Draxl}).
  \end{proof}

 \begin{prop}(\cite[Page 68]{Ershov}) 
 \label{exact} Suppose that $F$ is complete.
 Then the  complex
 $$
SK_1(D_1) \to SK_1(D)  \buildrel{\bar{\eta}} \over{\to}  G/H \to 1$$
is exact.  
\end{prop}

 \begin{proof} Since $\eta$ is surjective (\ref{eta-surj}), $\bar{\eta}$ is surjective. 
 
 Let $z \in SL_1(D_1)$.  Then Nrd$_{\bar{D}}(\bar{z}) =$ Nrd$_{\bar{D}_1}(\bar{z}) = 
 \overline{Nrd_{D_1}(z)} = 1$. Hence $\bar{\eta}(z) = 1$ and the image of $SK_1(D_1)$ in $SK_1(D)$ is in the kernel of 
 $\bar{\eta}$.  
 
 Let $z \in SL_1(D)$ with $\bar{\eta}(z) = 1$. 
 Then   $Nrd_{\bar{D}}(\bar{z}) = \theta^{-1} \tau(\theta)$ for some $\theta \in 
  Nrd_{\bar{D}}(\bar{D}^*)$. 
  Let $x \in \Lambda$ such that  $Nrd_{\bar{D}}(\bar{x}) = \theta$.
  Since $x^{-1} \pi_D x \pi_D^{-1} \in [D^*, D^*]$, replacing $z$ by 
  $z x \pi_D x^{-1} \pi_D^{-1} $, we assume that  $Nrd_{\bar{D}}(\bar{z}) = 1$. 
  Since $\bar{D} = \bar{D}_1$, $\bar{z} \in SL_1(\bar{D}_1)$. 
Since $D_1$ is unramified, the map $SL_1(D_1) \to SL_1(\bar{D}_1)$ is onto (\cite[Lemma 3.3]{Platonov1976})
and hence there exists $z_1 \in SL_1(D_1)$ such that $\bar{z}_1 = \bar{z}$.
Since $\bar{z}_1^{-1}\bar{z} = 1$, $z_1^{-1}z \in [D^*, D^*]$ (\ref{1modt}). 
In particular the $z_1 \mapsto z \in SK_1(D)$ and hence  the image of $SK_1(D_1)$ equal to ker$(\bar{\eta})$. 
 \end{proof}
 
 \begin{prop} 
 \label{sk1d-cd2}    Let $F$ be a complete discretely valued field with residue field $K$.
 Suppose that $K$ is a complete discretely valued field with residue field $\kappa$.
 Let $D$ be a central  simple  algebra over $F$ with ind$(A)$ coprime to char$(\kappa)$. 
 Suppose that $\bar{D}$ and   $Z(\bar{D})/K$  are unramified.
 If   cd$(\kappa) \leq 2$.
 Then  $SK_1(D) = \{1 \}$.
\end{prop}

 \begin{proof} Let $E \subseteq D$ be the lift of $Z(\bar{D})$ and $D_1 = C_D(E)$. Then,  
  by (\ref{exact}), we have 
  $$
 SK_1(D_1) \to SK_1(D)  \buildrel{\bar{\eta}} \over{\to}  G/H \to 1.$$ 
 Since $D_1$ is unramified, $SK_1(D_1) \simeq SK_1(\bar{D}_1)$ (\cite[Corollary 3.13]{Platonov1976}).
 Since cd$(\kappa) \leq 2$, by  (\cite{Platonov1978}), $SK_1(D_1)= \{ 1 \}$.
Let $\beta \in G$.  Then   $\beta  \in Z(\bar{D})^*$  with $N_{Z(\bar{D})/K}(\beta) = 1$. 
Since $Z(\bar{D})/K$ is unramified,  there exists $\theta \in Z(\bar{D})^*$ which is a unit
in the valuation ring of $Z(\bar{D})$ such that
$\beta = \theta^{-1} \tau(\theta)$.  Since $\bar{D}$ is unramified and 
cd$(\kappa) \leq 2$, $\theta$ is a reduced norm from $\bar{D}$ (\cite[Theorem 24.8]{Suslin1985}). Hence $G = H$
and $SK_1(D)$  is trivial. 
 \end{proof}

  \section{The group $SK_1(D)$ over a  two dimensional complete field}
 \label{sk1d-2dim}
 
 Let $R$ be a complete regular local ring of dimension 2, $m = (\pi, \delta)$ its maximal ideal, 
 $\kappa$ its residue field and 
 $F$ its field of fractions.  
 
 Let $D$ be a central division algebra over $F$ which is unramified on $R$ except possibly   
 at $(\delta)$.   Suppose that the period of  $D$ is coprime to char$(\kappa)$. 
 Then $D = D_0  +  (E, \sigma, \delta)$ for some central division algebra $D_0$ over
 $F$ which is unramified on $R$ and $E/F$ a cyclic extension which is unramified on $R$ (\ref{one}).
 
 Since $E/F$ is unramified on $R$, the integral closure $S$  of $R$ in $E$ is a complete regular 
 local ring with the maximal ideal $mS$. Let $E(\pi)$  be the  field of fractions of $S/(\pi)$. 
 Then $E(\pi)$ is the residue field of $E$ at $\pi$. Further $E(\pi)$ is a complete discrete valued field 
 with residue field $E_0 = S/mS$.

 Since $D_0$ is unramified on $R$, there exists an Azumaya algebra $\AA_0$ over $R$ such that 
 $\AA_0\otimes F \simeq D_0$.  
 Let $D_1$ be the central division algebra over $E$ which is Brauer equivalent to $D_0 \otimes_F E$. 
Then $D_1$ is unramified on $S$ and hence there exists an Azumaya algebra $\AA_1$ over $S$ such that 
$\AA_1 \otimes E \simeq D_1$. 
  
  Let $D(\pi)$ be the central division algebra over  $\kappa(\pi)$ which Brauer equivalent to 
  $(\AA_0\otimes_R \kappa(\pi)) \otimes (E(\pi), \sigma, \bar{\delta})$ and 
  $D_1(\pi) = \AA_1 \otimes _S E(\pi)$.  
  
  Let  $\bar{D}_1 = \AA_1 \otimes_S E_0  ( =  \AA_0 \otimes_R E_0) $.
  Let $G_0 = \{ a \in E_0^* \mid N_{E_0}/\kappa(a) = 1 \} \cap Nrd(\bar{D}_1)^*$ and 
  $H_0 = \{ \theta \tau(\theta)^{-1} \mid \theta \in Nrd(\bar{D}_1)^* \}$. 
By (\ref{exact}), we have the 
exact  sequence 
 $$ SK_1(D_1(\pi)) \to SK_1(D(\pi))  \buildrel{\bar{\eta}} \over{\to}  G_0/H_0 \to 1.$$

  Let $\Lambda_\pi$ be the maximal $\hat{R}_{(\pi)}$-order of $D\otimes F_\pi$.
  Since $D$ is unramified at $(\pi)$, we have $\Lambda_\pi/\pi\Lambda_\pi \simeq D(\pi)$. 
  Let $z \in SL_1(D\otimes F_\pi)$. Then $z \in \Lambda_\pi$.
  The map $z \mapsto \bar{z} \in  D(\pi)$ induces an isomorphism 
  $\phi : SK_1(D\otimes F_\pi) \to SK_1(D(\pi))$ (\cite[Corollary 3.13]{Platonov1976}).  
  
  Since $D_1$ is a central simple algebra over $E$ which is unramified on $S$ and $S$ is 
  a complete regular local ring we have an 
  isomorphism $\phi_1 : SK_1(D_1) \to SK_1(D_1(\pi))$ such that 
  if $x \in SL_1(\AA_1)$, then $\phi_1(x)$ is the class of the image of 
  $x$ in $\AA_1 \otimes _S E(\pi)$.

  \begin{lemma}
  \label{surj1-2dim}  The map $SK_1(D) \to SK_1(D \otimes F_\pi)$ is surjective.    
  \end{lemma}
  
\begin{proof} Let $z \in SL_1(D\otimes F_\pi)$.  Let $a  \in G_0$ with 
$\bar{\eta}(\phi(z)) = a\in G_0/H_0$.  Since $E/F$ is unramified on $R$ with residue field $E_0$, 
there exists $b \in E$  a lift of $a$ such that $N_{E/F}(b) = 1$. Since $\AA_1$ is unramified on $S$ and 
$a \in Nrd(\bar{D}_1)$ and $\bar{D}_1 =  \AA_1 \otimes_S E_0$, $b \in  Nrd (\AA_1) \subset Nrd(D_1)$. 
Let $z_1 \in D_1 $ with $Nrd_{D_1}(z_1) = b$. 
Since $D_1 \subseteq D$ and the center of $D_1$ is $E$, we have $Nrd_D(z_1) = N_{E/F}(Nrd_{D_1}(z_1)) = 1$. 
Since  $ \bar{\eta}(\phi(z_1) ) =  \bar{\eta}(\phi(z))$, replacing $z$ by $z_1^{-1}z$, we assume that 
$\bar{\eta}(\phi(z)) = 1$.  Hence there exists $x \in SK_1(D_1(\pi))$  which maps to $\phi(z)$ in $SK_1(D(\pi))$. 
Let $ y \in   SL_1( \AA_1)$  be such that   $\phi_1(x) = x$. 
Then $y\in SL_1(D)$ and maps to $z$ in $SK_1(D\otimes F_\pi)$.
\end{proof}

 Let $D$ be a central division algebra over $F$ which is unramified except possibly at $(\pi)$ and $(\delta)$.   
  Then $D\otimes F_\pi$ is a division algebra (\ref{division}).   
  Let $\Lambda_\pi = \{ x \in D \mid Nrd(x) \in R_{(\pi)} \}$.
  Then as in  (\S \ref{sk1d-dvr}), there exists $\pi_D \in \Lambda$ such that 
 $\bar{D}_\pi = \Lambda_\pi/\Lambda \pi_D$ is a division algebra over $\kappa(\pi)$. 
 Then there exists a cyclic unramified  extension $E_\pi/F_\pi$  with residue field $Z(\bar{D}_\pi)$ (cf. \S \ref{sk1d-dvr}). 
Then  by (\ref{splitting-field}), there exists a field extension $E/F$ unramified on $R$
except possibly at $(\pi)$ and $(\delta)$  such that  $E\otimes F_\pi \simeq E_\pi$
and $D \otimes E$ is split. In particular $E/F$ is unramified on $R$ except possibly at $(\delta)$.
Since $E_\pi/F_\pi$ is cyclic,   the proof of  (\ref{splitting-field}) also gives that 
$E/F$ is cyclic  and $E = E_1(\sqrt[e]{u\delta}) $ for some unramified extension $E_1/F$ and 
 $u$ a unit in the integral closure of 
$R$ in $E_1$. In particular  the integral closure $S$ of $R$ in $E$ is  a regular local ring with maximal ideal 
$(\pi, \delta')$  where $\delta' = \sqrt[e]{u\delta}$.  Let $\bar{S} = S/(\pi)$.  The $Z(\bar{D}_\pi)$ is the field of fractions of 
$\bar{S}$. 
 Let $D_1 = C_E(D)$.  Then $D_1$ is unramified on $S$ except possibly 
  at $(\delta')$.

 Let   $\tau$, $G$ and $H$  be as  in (\S \ref{sk1d-dvr})  defined with respect to the discrete valuation on $R$ given by 
 the prime ideal $(\pi)$. Then $\tau$ is a generator of Gal$(Z(\bar{D}/\kappa(\pi))$. Let $\tilde{\tau} \in Gal(E/F)$ be a lift of $\tau$. 
 Then $\tilde{\tau}$ is a generator of Gal$(E/F)$.

  \begin{prop} 
 \label{surj-2dim}  
 The map $\eta :   SL_1(D)  \to   G  $
is onto.  
\end{prop}

 \begin{proof}   Let $\beta \in G$.
  Since $N_{Z(\bar{D})/K}(\beta) = 1$, 
 there exists $\theta \in Z(\bar{D})$ such that $\beta = \theta^{-1}\tau(\theta)$. 
 We have  $\theta = w (\bar{\delta'})^i$ for some $w \in \bar{S} $ a unit and $i $ an integer.  
 Let $w_1 \in S$ be a lift of $w$ and $a = w_1 (\delta')^i$. Then $a$ is a lift of $\theta$. 
 Let    $b = a^{-1}\tilde{\tau}(a) \in E^*$.  Since $\tilde{\tau}(w_1) \in S$ is a unit and 
 $\tilde{\tau}( \delta') = \rho \delta'$ for some root of unity $\rho$,  
 $b \in S_1$ is a unit.   
 Since  $\beta$ the image  of $b$ is a reduced norm from $\bar{D}_\pi = \bar{D_1}_\pi$ and 
 $D_1 \otimes E_\pi$ is unramified, 
 $b$ is a reduced norm from $D_1 \otimes E_\pi$. Since $b \in S$ is a unit and $D_1$ is unramified on $S$ except possibly at 
 $(\pi)$ and $( \delta')$, by (\ref{reduced-norm}), $b$ is a reduced norm from $D_1$.
 Let $z \in D_1$ be such that Nrd$_{D_1}(z) = b$.  
 Since $Nrd_D(z) = N_{E/F}(Nrd_{D_1}(z)) = N_{E/F}(b) = N_{E/F}(a^{-1}\tilde{\tau}(a)) = 1$,
 $z \in SL_1(D)$.  Since $Nrd_{\bar{D}}(\bar{z}) = \beta$, $\eta$ is onto. 
 \end{proof}

  \begin{theorem} 
 \label{surj-2dim}  Let $R$ be a complete regular local ring of dimension 2, $m$ its maximal ideal, 
 $\kappa$ its residue field and 
 $F$ its field of fractions.  Let $D$ be a central division algebra over $F$ which is unramified on $R$ except possibly   
 at $(\pi)$ and $(\delta)$ for some primes with $m = (\pi, \delta)$.  
Then the map $SK_1(D) \to SK_1(D\otimes F_\pi)$ is surjective. 
\end{theorem}

 \begin{proof}  
  Suppose $D$ is unramified at $(\pi)$. Then, by (\ref{surj1-2dim}), $SK_1(D) \to SK_1(D\otimes_FF_\pi)$ is surjective. 
  
  Suppose that $D$ is ramified at $(\pi)$.  By (\ref{division}), $D \otimes F_\pi$ is division. 
  
  Since  $D$ is unramified on $R$ except possibly at  $(\pi)$ and $(\delta)$, we have 
 $E$, $D_1$, $G$ and $H$ be as  above  with respect to the valuation given by $(\pi)$. 
 Then the integral closure $S_1$ of $R$ in $E$ is a regular local ring with maximal
  ideal $(\pi, \delta')$ for some $\delta'$ lying over $\delta$. Further $D_{1}$ is unramified on $S_1$
  except possibly at  $(\delta')$. 
  
 We have the following commutative diagram  of complexes with bottom row exact (\ref{complex}, \ref{exact}) and the $\bar{\eta}$
  in the top row is 
 onto  (\ref{surj-2dim}).   
 $$
 \begin{array}{cccccc}
 SK_1(D_{1}) &  \to &  SK_1(D )  &  \buildrel{\bar{\eta}} \over{\to}  & G/H &  \to 1 \\
 \downarrow & & \downarrow& & \downarrow  \\

 SK_1(D_{1} \otimes F_\pi )  & \to & SK_1(D \otimes F_\pi ) &   \buildrel{\bar{\eta}} \over{\to}  &  G/H  & \to 1
 
 \end{array}
 $$
 
 Since $D_1$ is a central simple algebra over $E$, $E$ is the field of fractions of  a complete regular local ring $S_2$
 with maximal ideal $(\pi,\delta')$ and $D_1$ is unramified on $S_2$ except possibly at $(\delta')$, 
 by (\ref{surj1-2dim}), $SK_1(D_1) \to SK_1(D_1\otimes F_\pi)$ is surjective. 
 
 Hence a simple diagram chase gives the surjectivity of the middle map
  $SK_1(D) \to SK_1(D\otimes F_\pi)$. 
 \end{proof}
 
  \begin{theorem} 
 \label{2dim-sk1-cd2}  Let $R$ be a complete regular local ring of dimension 2, maximal ideal $m = (\pi, \delta)$, 
 $\kappa$ its residue field and 
 $F$ its field of fractions.  Let $\ell$ be a prime not equal to char$(\kappa)$.
  Let $D_0$ be a central division algebra over $F$ which is unramified on $R$ 
 and $E/F$ a cyclic extension  of degree $\ell^n$ which is unramified on $R$.  Let $\sigma$ be a generator of Gal$(E/F)$.
Let $D$ be a division algebra which Brauer equivalent to $ D_0 + (E, \sigma, \pi^i\delta^j)$ for some $0 \leq i, j \leq \ell^n-1$. 
Suppose that cd$(\kappa) \leq 2$ and per$(D_0)$ is coprime to char$(\kappa)$.
Then either $SK_1(D \otimes F_\pi)$ or $SK_1(D\otimes F_\delta)$ is trivial. 
In particular   the map 
$$SK_1(D) \to SK_1(D\otimes F_\pi) \times SK_1(D\otimes F_\delta)$$ is surjective. 
\end{theorem}

 \begin{proof}  
  
 Suppose $\ell$ divides both $i$ and $j$.  Write $i  = \ell^d i_0$ and $j = \ell^dj_0$ for some $d \geq 1$ and
$\ell$ does not divide at least one of $i_0$ and $j_0$. Then 
$$(E, \sigma, \pi^i \delta^j) = (E, \sigma, (\pi^{i_0}\delta^{j_0})^{\ell^d}) = (E, \sigma)^{\ell^d} \cdot (\pi^{i_0} \delta^{j_0})
= (E^{\sigma^{\ell^{n-d}}}, \sigma \mid_{E^{\sigma^{\ell^{n-d}}}}, \pi^{i_0} \delta^{j_0}).$$
Thus replacing $(E, \sigma)$ by $(E^{\sigma^{\ell^{n-d}}}, \sigma \mid_{E^{\sigma^{\ell^{n-d}}}})$,
$i$ by $i_0$ and $j$ by $j_0$, we assume that either $i$ or $j$ is not divisible by $\ell$. 

Suppose $\ell$ does not divide $i$. Then there exists $1 \leq i' \leq \ell^n-1$ such 
 $ii' = 1$ modulo $\ell^n$. We have 
 $$  (E, \sigma, \pi^i\delta^j) =  ii'(E, \sigma, \pi^i\delta^j) =
  (E, \sigma^i, \pi \delta^{i'j}).$$ Hence replacing $\sigma$ by $\sigma^i$ and $j$ by $i'j$, 
  we assume that $i = 1$. 
  
Since $D$ is unramified on $R$ except possibly at $(\pi)$ and $(\delta)$, $D \otimes F_\pi$ is a division algebra  (\ref{division}).
Let $\Lambda_\pi$ be the maximal $\hat{R}_\pi$-order of $D\otimes F_\pi$. 
Since $E/F$ is unramified on $R$, the integral closure $S$ of $R$ in $E$ is a  regular local ring with maximal ideal 
$(\pi, \delta)$.   Since $D_0$ is unramified on $R$, there exists am Azumaya $R$-algebra $\DD_0$ such that 
$\DD_0 \otimes F \simeq D_0 $  (cf. \cite[Lemma 3.1]{LPS}).  Since $i = 1$, $Z(\bar{D}_\pi)$ is the field of fractions of $S/(\pi)$
and   $\bar{D}_\pi =  ( \DD_0 \otimes S/(\pi)) \otimes Z(\bar{D}_\pi)$.  In particular 
$\bar{D}_\pi $ is unramified on $S/(\pi)$.  Thus, $SK_1(D\otimes F_\pi)$ is trivial (\ref{sk1d-cd2}). 

Since $SK_1(D) \to SK_1(D\otimes F_\delta)$ is onto (\ref{surj-2dim}), 
$$SK_1(D) \to SK_1(D\otimes F_\pi) \times SK_1(D\otimes F_\delta)$$ is surjective.   
 \end{proof}

 \section{The group $SUK_1(A, \tau)$ over a discrete valued field }
  \label{sk1ud-dvr}
 
 Let $F_0$ be a field and  $F/F_0$ be a quadratic  \'etale extension.
 Let $A$ be a central simple algebra over $F$ with a $F/F_0$-involution $\tau$.  
 Let $Sym(A,\tau) = \{ x \in A \mid \tau(x)= x \}$ and  let $\Sigma(A, \tau)$
 be the subgroup of $A^*$ generated by $Sym(A, \tau)^*$. 
 Then for any $z \in \Sigma(A,\tau)$, $Nrd_A(z) \in F_0$.  
  Let 
 $$\Sigma'(A, \tau)  = \{  z  \in B^* \mid Nrd_A(z) \in F_0^*\}.$$
 Then $\Sigma(A, \tau) \subseteq \Sigma'(A, \tau)$. 
 Let 
 $$SKU_1(A, \tau) = \Sigma'(A,\tau) / \Sigma(A, \tau) .$$
 
 The reduced norm homomorphism gives a  surjective homomorphism   
 $$SUK_1(A,\tau) \to F_0^* \cap Nrd_A(A^*)/ Nrd_A( \Sigma(A, \sigma)). $$
 Since $SL_1(A) \subset \Sigma'(A,\tau)$ and  $[A^*, A^*] \subset \Sigma(A, \tau)$ (cf. \cite[17.26]{KMRT}), 
 we get complex 
  $$
 SK_1(A) \to SUK_1(A, \tau) \to F_0^* \cap Nrd_A(A^*)/ Nrd_A( \Sigma(A, \sigma)) \to 1$$
 
 Let $z \in \Sigma'(A, \tau)$. Suppose that Nrd$(z) \in \Sigma(A, \tau)$.
 Then there exists $x \in \Sigma(A, \tau)$ such that Nrd$(z) = $ Nrd$(x)$.
 Then $x^{-1}z \in SL_1(A)$ and maps to $z$ in $SUK_1(A, \tau)$. Hence the above complex is exact.

 Suppose $F_0$ is a   complete discretely valued field  with residue field $K_0$
 and $F/F_0$ a quadratic separable extension. 
 Let $D$ be a central division  algebra over $F$ with  a $F/F_0$-involution $\tau$.
 Let $R_0$ be the valuation ring of $F_0$ and $R$ the integral closure of $R_0$ in $F$. 
 Let $\Lambda$ be the maximal $R$-order in $D$ and $m_D$ the unique maximal ideal of
 $\Lambda$.  Let $\bar{D} = \Lambda/m_D$.  Suppose that 
 2ind$(D)$ is coprime to char$(K_0)$. 
  Suppose that $F/F_0$ is unramified.  Let $K$ be the residue field of $F$.
 Then $Z(\bar{D})/K$ is a cyclic extension. Let $\sigma$ be a generator of Gal$(Z(\bar{D}/K)$. 

There exist parameter $\pi_D \in \Lambda$ such that $\tau(\pi_D) = \pi_D$.
Then $\tau$ induces a involution $\bar{\tau} $ on $\bar{D}$ of second kind. 
Let $\tau' = int(\pi_D) \tau$. Then $\tau'$ is also a $F/F_0$-involution on $D$
and induces an involution $\bar{\tau}'$ on $\bar{D}$. 
We recall the following from (\cite[\S 4]{Y1979}).  An element $a \in Z(\bar{D})^*$ is called a 
{\it projective unitary conorm} if there exist $b \in Nrd_{\bar{D}}(\bar{D}^*)$ such that 
$\sigma(a) a^{-1} = b\bar{\tau}'(b)^{-1}$. Let $P(A, \tau) $ be the set of all projective unitary norms.
Then $K^*Nrd_{\bar{D}}(\Sigma(\bar{D}, \bar{\tau})) \subset P(A, \tau)$.
Let $PU(A, \tau) = P(A, \tau)/ K^*Nrd_{\bar{D}}(\Sigma(\bar{D}, \bar{\tau})) $. 

 Then by (\cite[Theorem 4.12]{Y1979}), we have an exact sequence 
 $$ SUK_1(\bar{D}, \bar{\tau}') \to SUK_1(A, \tau) \to PU(A, \tau)\to 1.$$
  
 \begin{theorem}
 \label{sk1u-dvr}  Let $K_0$ be a complete discretely  valued field with residue field $\kappa_0$ and 
 $K/K_0$ an unramified quadratic  extension.   Let $F_0$ be a complete discretely valued field  with residue field $K_0$
  and  $F/F_0$ a unramified   quadratic  extension with residue field $K$. Let $\pi \in F_0$ be a parameter. 
 Let $D/F$ be a central division  algebra with an $F/F_0$-involution $\tau$. 
 Suppose that $\bar{D}$ and $Z(\bar{D})/K$ are  unramified. Suppose that 
 2per$(A)$ is coprime to char$(\kappa_0)$. 
 If   cd$(\kappa_0) \leq 2$, then  $SUK_1(A, \tau)$ is trivial. 
 \end{theorem}
 
 \begin{proof}   
 Since $\bar{D}$ is unramified and cd$(\kappa_0) \leq 2$, 
 $SUK_1(\bar{D}, \bar{\tau})$ and $SUK_1(\bar{D}, \bar{\tau}')$  are  trivial. 
 
 Let $a \in  P(A, \tau)$. Then, by (\ref{conorms} ), 
 $a \in K^* Z(\bar{D})^{\bar{\tau}}$.  Let $\delta \in K_0$ be a parameter.
 Since $K/K_0$ and $Z(\bar{D})/K$ is unramified, $\delta \in Z(\bar{D}^{\bar{\tau}})$ is a parameter. 
 Since every element in $Z(\bar{D}^{\bar{\tau}}$ is a product of power of $\delta$ and a unit,
 we can write $ a= bc$ for some $c \in Z(\bar{D}^{\bar{\tau}}$ which is a unit. 
 Since $\bar{D}$  is unramified and cd$(\kappa_0) \leq 2$, $c \in \Sigma'(\bar{D}, \bar{\tau})$. 
Since $SUK_1(\bar{D}, \bar{\tau})$ is trivial, $  \Sigma'(\bar{D}, \bar{\tau}) =  \Sigma(\bar{D}, \bar{\tau})$
and hence $a$ is trivial in $PU(A, \tau)$.

 Hence,  by exact sequence   
 $$ SUK_1(\bar{D}, \bar{\tau}') \to SUK_1(A, \tau) \to PU(A, \tau)\to 1, $$
 $SUK_1(A, \tau)$ is trivial. 
 \end{proof}

  \section{The group $SUK_1(A, \tau)$ over a two dimensional  complete  fields }
   \label{sk1ud-2dim}

 Let $R_0$ be a complete regular local ring with maximal ideal $(\pi, \delta)$, residue field $\kappa$ and field of 
fraction $F_0$. Let $F/F_0$ be a quadratic  field  extension.  
Suppose that char$(\kappa) \neq 2$.
Then $F = F(\sqrt{d})$ for some $d \in F_0^*$. 
Suppose that $d = u a^2$ or $\pi  a^2$  or $\delta a^2$ for some $u \in R$ a unit and $a \in F^*$. 
Let $R$ be the integral closure of $R_0$ in $F$. Then 
$R$ is a complete regular local ring with maximal ideal $(\pi', \delta')$ for some primes $\pi'$ and 
$\delta'$ lying over $\pi$ and $\delta$ respectively.

 \begin{theorem}
 \label{sk1u-hyp2} Let $A/F$ be  a central simple algebra   with a $F/F_0$-involution $\tau$.
Suppose that 2 and  period  of $A$ are  coprime to char$(\kappa)$. Suppose that 
$A$ is unramified on $R$ except possibly at $(\pi')$ and $(\delta')$. 
Then the map 
 $$ SUK_1( A, \tau)(F_P)   \to SUK_1( (( A, \tau) \otimes F_{0\pi}))$$
 is onto.  
 \end{theorem}
 
 \begin{proof}   We have the following commutative diagram  $$
 \begin{array}{cccccc} 
SK_1(A) & \to &  SUK_1(A, \tau) &  \to &  \frac{F_0^* \cap Nrd_A(A^*)}{Nrd_A( \Sigma(A, \sigma)) }& \to 1 \\
\downarrow & & \downarrow & & \downarrow \\
SK_1(A \otimes F_{0\pi} ) & \to &  SUK_1((A, \tau) \otimes F_{0\pi})  & \to &  \frac{ F_{0\pi}^* \cap Nrd_A((A\otimes F_{0\pi})^*)}{
 Nrd_A( \Sigma(A \otimes F_{0\pi}, \sigma))} & \to 1 
 \end{array}
$$
with   exact rows (cf. \S \ref{sk1ud-dvr}).  

Let $a \in F_{0\pi}^*$.  Let $n = ind(A)$. Then $a = w \pi^r\delta^s b^n$ for some 
$w \in R_0$ a unit, $b \in F_0$ and $r, s \in \Z$ (cf., \cite[Remark 7.1]{PS2022}). 
Since $b^n \in Nrd(\Sigma((A\otimes F_{0\pi}), \sigma))$, we have $a =  w \pi^r \delta^s
 \in F_{0\pi}^* \cap Nrd_A((A\otimes F_{0\pi})^*)/ Nrd_A( \Sigma(A \otimes F_{0\pi}, \sigma)).$
Since $w \pi^r \delta^s \in F_0$ and $w \pi^r \delta^s$ is a reduced from $A\otimes F_{0\pi}$, 
$w \pi^r \delta^s $ is a reduced norm from $A$ (\ref{reduced-norm}). Hence the right  vertical map in the above diagram is onto.

Since the first vertical map is onto (\ref{surj-2dim}), the middle vertical map is onto. 
 \end{proof}
 
  \begin{theorem} 
 \label{2dim-sk1u-cd2}    
  Let $A_0$ be a central simple algebra over $F$ which is unramified on $R$. 
 and $E/F$ a cyclic extension    which is unramified on $R$.  Let $\sigma$ be a generator of Gal$(E/F)$.
Let $A =  A_0 + (E, \sigma, \pi^i\delta^j)$ for some $0 \leq i, j \leq \ell^n-1$. 
Suppose that cd$(\kappa) \leq 2$ and per$(A)$ is coprime to char$(\kappa)$.
If $A$ admits a $F/F_0$-involution $\tau$, then 
 either $SUK_1(A \otimes F_{0 \pi}, \tau)$ or $SUK_1(A\otimes F_{0\delta}, \tau)$ is trivial. 
In particular   the map 
$$SUK_1(A, \tau) \to SUK_1(A\otimes F_{0\pi}, \tau) \times SUK_1(A\otimes F_{0\delta}, \tau)$$ is surjective. 
\end{theorem}

 \begin{proof}  As in the proof of  (\ref{2dim-sk1-cd2}),  we can assume that 
 $i = 1$ or $j = 1$.  Say $i = 1$.
  Let $D$ and $D_0$ be the division algebras Brauer equivalent to $A$ and $A_0$ respectively. 
  Then $Z(\bar{D \otimes F_\pi})$ is the residue field of $E$ at $ \pi$ and hence 
 unramified over the residue field of $K$ at $\pi$.
  Further  $\bar{D\otimes F_\pi}$ is Brauer equivalent to $\bar{D}_0 \times Z(\bar{D \otimes F_\pi})$ and hence
  unramified. Therefore, by (\ref{sk1u-dvr}), $SUK_1(D\otimes F_\pi, \tau)$ is trivial. 
 \end{proof}

 \section{Groups of Type $^1A_n$} 
  \label{type-inner}

\begin{theorem}
\label{sk1-requiv} Let $T$ be a complete discrete valuation ring with residue field $\kappa$ and field of 
fraction $K$. Let $F$ be the function field of a curve over $K$. 
Let $A/F$ be  a central simple algebra  of  period  coprime to char$(\kappa)$.
Let $\XX_0$ be a regular proper model of $F$ such that  the union of 
the  ramification locus of $A$ and the closed fibre  $X_0$
is a union of regular curves with normal crossings. 
Let $\PP_0$ be a finite set of closed points of $\XX$ containing all the singular points of $X_0$
and at least one closed point from each irreducible component of $X_0$.  Let $\UU_0$ be the set
of irreducible components of $X \setminus \PP_0$ and $\BB_0$ the set of branches with respect to $\PP_0$.
Then 
 $$\Sha_{div}(F,  SL_1(A)) \simeq  \prod_{U \in \UU_0} SK_1(A)(F_U) \,\backslash \prod_{\wp \in \BB_0} 
 SK_1(A)(F_\wp) \,/ \prod_{P \in \PP_0} SK_1(A)(F_{P}),$$
\end{theorem}

\begin{proof}  By (\ref{sl1-hyp1}),  the group $ SL_1(A)$ satisfies the hypothesis (\ref{hyp1}) and 
by (\ref{surj-2dim}), $SL_1(A)$ satisfies the hypothesis (\ref{hyp2}). 
Since $H^1(L, SL_1(A)) \simeq H^1(L, SL_1(M_m(A))$ for all extensions $L/F$ and $m \geq 1$, 
we assume that $A$ is not division. Hence the group $SL_1(A)$ is isotorpic.
Thus the theorem follows from (\ref{dvrsha-requiv}). 
\end{proof}

\begin{cor}
\label{good-reduction-sl1} Let $T$ be a complete discrete valuation ring with residue field $\kappa$ and field of 
fraction $K$.   Let $X$ be a smooth projective curve over $K$ which has a good reduction. 
Let $F$ be the function field  a curve over $K$ and $\XX$ a smooth projective model of $X$.   
Let $A$ be a central simple algebra over $F$ which is unramified on $\XX$ except possibly at the special fibre. 
Then $\Sha_{div}(F, SL_1(A)) = \{ 1 \}$.
\end{cor}

\begin{proof} Since $\XX$ is a smooth model, the special fibre is a smooth irreducible curve. 
Let  $\PP_0$ be any nonempty finite subset of closed points of $\XX$.
Then, by \ref{sk1-requiv}, we have 
 $$\Sha_{div}(F,  SL_1(A)) \simeq  \prod_{U \in \UU_0} SK_1(A)(F_U) \,\backslash \prod_{\wp \in \BB_0} 
 SK_1(A)(F_\wp) \,/ \prod_{P \in \PP_0} SK_1(A)(F_{P}).$$
 Since there are no singular points on the special fibre, there is a bijection 
 between $\PP_0$ and the branches $\BB_0$.
 Since $SK_1(A)(F_P) \to SK_1(A)(F_\wp)$ is onto (\ref{surj1-2dim}) the branch $\wp$ at $P$, for all $P \in \PP_0$,
 it follows that $\Sha_{div}(F, SL_1(A)) = \{1 \}$.  
\end{proof}

\begin{cor}
\label{good-reduction-sl11} Let $T$ be a complete discrete valuation ring with residue field $\kappa$ and field of 
fraction $K$. Let $X$ be a smooth projective curve over $K$ which has a good reduction. 
Let $F$ be the function field $X$. Lee Let $A$ be a central simple algebra over $K$.
Then $\Sha_{div}(F, SL_1(A)) = \{ 1 \}$.
\end{cor}

\begin{proof} Since $X$ has a good reduction, there exists a smooth proper model $\XX$  of $F$.
Since $A$ is defined over $K$, $A$ is unramified on $\XX$ except possibly at the special fibre of $\XX$.
Hence, by (\ref{good-reduction-sl1}),  $\Sha_{div}(F, G) = \{1\}$. 
\end{proof}

\section{Local-global principle for Unitary groups} 
 \label{lgp-unitary}

  Let $T$ be a complete discrete valuation ring with residue field $\kappa$ and field of 
fraction $K$. Let $F$ be the function field of a curve over $K$. 
Let $A/F$ be  a central simple algebra   with an involution $\tau$  of any kind.
Let $F_0 = F^\tau$. 
Suppose that $2per(A)$ is  coprime to char$(\kappa)$. 

\begin{prop}  
\label{unitary}    $\Sha_{div}(F_0, U(A, \tau))  = \{ 1 \}$.
 \end{prop}
  \begin{proof}   Let $\zeta \in H^1(F_0, U(A, \tau))$.  Since $H^1(F_0, U(A, \tau))$  is in bijection with
 the set of isomorphism classes of   involutions on  $A$, $\zeta$ defines an involution $\tau'$ such that 
 $(A, \tau')  \otimes F_\nu \simeq  (A, \tau) \otimes F_\nu$ for all $\nu \in \Omega_F$.
 Hence,  by  (\cite{GP}),  $(A,\tau') \simeq (A, \tau)$. Thus $\Sha_{div}(F_0, U(A, \tau))  = \{ 1 \}$. 
 
 \end{proof}
 
\begin{cor}
\label{unitary-hyp1}  Let $\XX_0$ be a regular proper model of $F_0$. 
Then for every closed point $P$ of $\XX_0$,  $\Sha_{div}(F_{0P}, U(A,\tau)) = \{ 1\}$.
 \end{cor}
 
 \begin{proof} Let  $X_0$ be the  closed fibre of $\XX_0$.
 Then, by (\ref{exact-seq}), we have the following exact sequence 
 $$1 \to \Sha_X(F_0, U(A, \tau)) \to \Sha_{div}(F_0, U(A,\tau))  \to \prod_{P\in X_{(0)}}^{\prime} \Sha_{div}(F_{0P}, U(A, \tau)) \to 1.$$ 
  By  (\ref{unitary}),   $\Sha_{div}(F_0, U(A,\sigma)) = \{ 1\}$.
  Hence, we have $\Sha_{div}(F_{0P}, U(A, \tau))  = \{ 1 \}$ for every closed point 
$P$ of $\XX_0$.   
\end{proof}

Suppose $\tau$ is of second kind.  Write $F = F_0(\sqrt{d})$ for some $d \in F_0^*$. 
 Let $\XX_0$ be a regular proper model of $F_0$ such that the 
 ramification locus of $(A,\tau)$, the closed fibre  and the support of $d$ is a union of regular curves with normal crossings.
 Further we assume that  for every closed point $P \in \XX_0$,   the maximal ideal  at $P$ is given by $(\pi_P, \delta_P)$ 
 such that $d = u_P f^2$ or $ u_P\pi_P f^2 $ or $u_P\delta_Pf^2$  for some unit $u_P$ at $P$ and
 $f \in F_P^*$.  We note that such a model exists (\cite[Theorem 11.2]{PS2022}). 
 
  \begin{prop}  
\label{su-hyp1} The group $SU(A,\tau)$ satisfies the local injectivity hypothesis (\ref{hyp1}) with respect to $\XX_0$.
 \end{prop}
 
 \begin{proof}  Let $P \in \XX_0$ be a closed point. 
 Let $\zeta \in H^1(F_{0P}, SU(A, \tau))$ which maps to the trivial element in $H^1(F_{0P\nu}, SU(A, \tau))$ for all divisorial discrete valuations of 
 $F_{0P}$.  Then, by (\ref{unitary-hyp1}), the image of $\zeta \in  H^1(F_{0P}, U(A, \tau))$ is trivial.
 Hence, by (\ref{su-hyp11}), $\zeta$ is trivial. 
 \end{proof}

 \section{Groups of Type $^2A_n$} 
  \label{type-outer}

\begin{theorem} 
\label{su-iso} Let $T$ be a complete discrete valuation ring with residue field $\kappa$ and field of 
fraction $K$. Let $F$ be the function field of a curve over $K$. 
Let $A/F$ be  a central simple algebra   with an involution $\tau$ of second kind.
Suppose that $(A,\tau)$ is isotropic. 
Suppose that 2 and  period  of $A$ are  coprime to char$(\kappa)$. 
Let $\XX_0$ be a regular proper model of $F$ with the union of 
the  ramification locus of $(A,\tau)$ and the closed fibre  $X_0$
is a union of regular curves with normal crossings. 
Let $\PP_0$ be a finite set of closed points of $\XX$ containing all the singular points of $X_0$
and at least one closed point from each irreducible component of $X_0$. 
 Let $\UU_0$ be the set
of irreducible components of $X \setminus \PP_0$ and $\BB_0$ the set of branches with respect to $\PP_0$.
Let $F_0 = F^\tau$. 
Then 
 $$\Sha_{div}(F_0,  SU(A, \tau)) \simeq  \prod_{U \in \UU_0} K_1SU(A, \tau)(F_U) \,
 \backslash \prod_{\wp \in \BB_0} 
 K_1SU(A, \tau)(F_\wp) \,/ \prod_{P \in \PP_0}K_1SU(A, \tau)(F_{P}).$$  
 \end{theorem}
 
 \begin{proof}  By (\ref{su-hyp1}),  the group $ SU(A, \tau)$ satisfies the hypothesis (\ref{hyp1}) and 
by (\ref{sk1u-hyp2}) satisfies the hypothesis (\ref{hyp2}). 
 Since $(A,\tau)$ is isotropic, the group   $SU(A, \tau)$   is isotropic. 
  Hence the theorem follows from (\ref{dvrsha-requiv}). 
\end{proof}

Let $L$ be a field and $E/L$ a quadratic \'etale extension. Let $L^{*1} = \{ a \in L^* \mid N_{E/L}(a) = 1\}$.
Let $A$ be a central simple algebra over $E$ with a $E/L$-involution $\sigma$. 
Let $H_A = \{ \theta \sigma(\theta)^{-1} \mid \theta \in Nrd(A) \} \subseteq L^{*1}$. 

\begin{prop}  
\label{su-compu} Let $F$  be as in (\ref{su-iso}). 
Let $A$ be a central simple algebra over $F$ with an involution $\sigma$ of second kind. Let $F_0 = F^\tau$.
Then 
$$\Sha_{div}(F_0, SU(A, \sigma)) \simeq ker ( F^{*1}/H_A \to \prod_{\nu \in \Omega_{F_0}} F_\nu^{*1}/ H_{A\otimes F_{0\nu}}),$$
where  $F_\nu = F\otimes F_{0\nu}$. 
\end{prop}

\begin{proof}  The short exact sequence of algebraic groups 
$$ 1 \to SU(A, \sigma) \to U(B, \sigma) \to R^1_{F/F_0}\G_m \to 1$$
gives rise to the following commutative diagram of exact rows

$$
\begin{array}{cccccccc}
U(A, \tau)(F_0) &  \to &  F^{*1} & \to & H^1(F_0, SU(A, \sigma)) &  \to &  H^1(F_0, U(A, \sigma)) \\

\downarrow & & \downarrow & & \downarrow & & \downarrow \\

\prod_\nu U(A, \tau)(F_{0\nu})  & \to &\prod_\nu  F_\nu^{*1}  & \to & \prod_\nu H^1(F_{0\nu}, SU(A, \sigma)) & \to & 
\prod_\nu H^1(F_{0\nu}, U(A, \sigma)).
\end{array} 
$$
Hence, by the definition of the subgroups $H_A$, we have 
the following commutative diagram  of exact rows 
$$
\begin{array}{cccccccc}
 1  &  \to &  F^{*1}/ H_A  & \to & H^1(F_0, SU(A, \sigma)) &  \to &  H^1(F_0, U(A, \sigma)) \\

 & & \downarrow & & \downarrow & & \downarrow \\

1  & \to &\prod_\nu  F_\nu^{*1} / H_{A\otimes F_{0\nu}} & \to & \prod_\nu H^1(F_{0\nu}, SU(A, \sigma)) & \to & 
\prod_\nu H^1(F_{0\nu}, U(A, \sigma)). 
\end{array} 
$$
Since the last vertical arrow in the above diagram is injective (\ref{unitary}), the result follows. 
\end{proof}

 Now we prove a   result similar to (\ref{su-iso}) without the assumption that $(A, \tau)$ is isotropic. 
 
\begin{theorem} 
\label{su}
Let $T$ be a complete discrete valuation ring with residue field $\kappa$ and field of 
fraction $K$. Let $F$ be the function field of a curve over $K$. 
Let $A/F$ be  a central simple algebra   with an involution $\tau$ of second kind.
Suppose that 2 and  period  of $A$ are  coprime to char$(\kappa)$. 
Let $\XX_0$ be a regular proper model of $F$ with the union of 
the  ramification locus of $(A,\tau)$ and the closed fibre  $X_0$
is a union of regular curves with normal crossings. 
Let $\PP_0$ be a finite set of closed points of $\XX_0$ containing all the singular points of $X_0$
and at least one closed point from each irreducible component of $X_0$. 
 Let $\UU_0$ be the set
of irreducible components of $X_0 \setminus \PP_0$ and $\BB_0$ the set of branches with respect to $\PP_0$.
Let $F_0 = F^\tau$.
Then 
 $$\Sha_{div}(F_0,  SU_1(A, \tau)) \simeq  \prod_{U \in \UU_0} K_1SU(A, \tau)(F_U) \,
 \backslash \prod_{\wp \in \BB_0} 
 K_1SU(A, \tau)(F_\wp) \,/ \prod_{P \in \PP_0}K_1SU(A, \tau)(F_{P}).$$  
\end{theorem}

 \begin{proof} Let $A'$ be central simple algebra over $F$ which is Brauer equivalent  to $A$
 and $\tau'$ a $F/F_0$-involution.  
  Since for any  field extension $L/F_0$, 
   Nrd$(A \otimes L) = $ Nrd$(A' \otimes L)$,  we have $H_{A\otimes L} = H_{A' \otimes L}$. 
   Hence, by (\ref{su-compu}), 
    $$\Sha_{div}(F_0, SU(A, \sigma)) \simeq \Sha_{div}(F_0, SU(A', \tau').$$
    
  Let $A' = M_2(A)$ and $\tau'$ the adjoint involution on $A'$ given by the hyperbolic form. 
  Then, $SU(A', \tau')$ is isotropic. 
 Since for any  field extension $L/F_0$,  $K_1SU(A,\tau)(L) \simeq K_1SU(A', \tau')$ (\cite[Lemma 3]{Y1974}), 
 the result follows from (\ref{su-iso}).
 \end{proof} 
 
Let $L$ be a field and $E/L$ a quadratic \'etale extension. Let $L^{*1} = \{ a \in L^* \mid N_{E/L}(a) = 1\}$.
Let $A$ be a central simple algebra over $E$ with a $E/L$-involution $\sigma$. 
Let $H_A = \{ \theta \sigma(\theta)^{-1} \mid \theta \in Nrd(A) \} \subseteq L^{*1}$.

  \begin{cor}
\label{good-reduction-su} Let $T$ be a complete discrete valuation ring with residue field $\kappa$ and field of 
fraction $K$.  Let $K/K_0$ be a quadratic  separable extension.  Let $A$ be a central simple algebra over 
$K$ with a $K/K_0$-involution.  Suppose that 2ind$(A)$ is a coprime to char$(\kappa)$.
Let $X$ be a smooth projective curve over $K$ which has a good reduction. 
Let $F_0$ be the function field    $X$ and $F = K \otimes F_0$. 
Then $\Sha_{div}(F_0, SU(A,\tau)) = \{ 1 \}$.
\end{cor}

\begin{proof} Since  $X$ has a good reduction, there exists a  smooth  projective model $\XX$ of $F$.
Let $X_0$ be the special fibre of $\XX$. Then $X_0$ is a smooth irreducible curve. 
Let  $\PP_0$ be any nonempty finite subset of closed points of $\XX$.
Then, by \ref{su}, we have 
 $$\Sha_{div}(F_0,  SU(A, \tau)) \simeq  \prod_{U \in \UU_0} K_1SU(A, \tau)(F_U) \,
 \backslash \prod_{\wp \in \BB_0} 
 K_1SU(A, \tau)(F_\wp) \,/ \prod_{P \in \PP_0}K_1SU(A, \tau)(F_{P}).$$  
 Since there are no singular points on the special fibre, there is a bijection 
 between $\PP_0$ and the branches $\BB_0$.
 Since $SUK_1(A, \tau)(F_P) \to SK_1(A, \tau)(F_\wp)$ is onto (\ref{sk1u-hyp2}) 
 for  the branch $\wp$ at $P$, for all $P \in \PP_0$,
 it follows that $\Sha_{div}(F, SU(A, \tau)) = \{1 \}$.  
\end{proof}

  \section{Cohomological dimension 1}
  \label{cd1}
 
 In this section we prove our main theorem when the residue field has cohomological dimension at most 1.

 \begin{theorem}
 \label{sk1-cd1} Let $T$ be a complete discrete valuation ring with residue field $\kappa$ and field of 
fraction $K$. Let $F$ be the function field of a curve over $K$. 
Let $A/F$ be  a central simple algebra  of  period  coprime to char$(\kappa)$.
If cd$(\kappa) \leq 1$, then   $\Sha_{div}(F, SL_1(D))$ is trivial. 
\end{theorem}

\begin{proof}
Let $\XX_0$ be a regular proper model of $F$ such that  the union of 
the  ramification locus of $A$ and the closed fibre  $X_0$
is a union of regular curves with normal crossings. 
Let $\PP_0$ be a finite set of closed points of $\XX$ containing all the singular points of $X_0$
and at least one closed point from each irreducible component of $X_0$.  Let $\UU_0$ be the set
of irreducible components of $X \setminus \PP_0$ and $\BB_0$ the set of branches with respect to $\PP_0$.
Then, by (\ref{sk1-requiv}), we have. 
 $$\Sha_{div}(F,  SL_1(A)) \simeq  \prod_{U \in \UU_0} SK_1(A)(F_U) \,\backslash \prod_{\wp \in \BB_0} 
 SK_1(A)(F_\wp) \,/ \prod_{P \in \PP_0} SK_1(A)(F_{P}).$$
 Let  $\wp \in \BB$ be a branch and $\kappa(\wp)$ be the residue field of $F_\wp$.
 Then $\kappa(\wp)$   a complete discrete valued field with residue field a finite extension of 
 $\kappa$. Since cd$(\kappa) \leq 1$, cd$(\kappa(\wp) \leq 2$ (\cite[p. 85, Proposition 12]{SerreGC}). Hence,
by (cf. \cite{Soman}), 
 $SK_1(A)(F_\wp)$ is trivial. In particular $\Sha(F, SL_1(D))$ is trivial. 
\end{proof}

\begin{theorem}
 \label{sk1u-cd1} Let $T$ be a complete discrete valuation ring with residue field $\kappa$ and field of 
fraction $K$. Let $F_0$ be the function field of a curve over $K$ and $F/F_0$ a  quadratic  field extension. 
Let $A/F$ be  a central simple algebra  over $F$ of period  $n$ with a $F/F_0$-involution $\tau$.
Suppose that $2n$  is   coprime to char$(\kappa)$. 
If cd$_n(\kappa) \leq 1$, then   $\Sha_{div}(F_0, SU_1(A, \tau))$ is trivial.  
 \end{theorem}

 \begin{proof}
 Let $\XX_0$ be a regular proper model of $F_0$ with the union of 
the  ramification locus of $(A,\tau)$ and the closed fibre  $X_0$
is a union of regular curves with normal crossings. 
Let $\PP_0$ be a finite set of closed points of $\XX$ containing all the singular points of $X_0$
and at least one closed point from each irreducible component of $X_0$. 
 Let $\UU_0$ be the set
of irreducible components of $X \setminus \PP_0$ and $\BB_0$ the set of branches with respect to $\PP_0$.
Then, by (\ref{su}), 
 $$\Sha_{div}(F_0,  SU(A, \tau)) \simeq  \prod_{U \in \UU_0} K_1SU(A, \tau)(F_U) \,
 \backslash \prod_{\wp \in \BB_0} 
 K_1SU(A, \tau)(F_\wp) \,/ \prod_{P \in \PP_0}K_1SU(A, \tau)(F_{P}).$$  
 Let $\wp \in \BB_0$. Since $F_\wp$ is a complete discretely valued field with  cohomological dimension of
 the residue field   $\kappa(\wp)$ is at most 2, by  (\cite[Corollary 4.15]{Y1979}), $K_1SU(A, \tau)(F_\wp)$ is trivial. Hence 
  $\Sha_{div}(F, SU_1(A, \tau))$ is trivial.   
 \end{proof}

  \begin{cor} \label{an-cd1}
  Let $K$ be a complete discretely  valued  field  with residue field $\kappa$. 
  Let $F$ be the function field of a curve over $K$.  Let $G$ be a semisimple simply connected group  over $F$ of type
  $^iA_n$.  If   cd$_{n+1}(\kappa) \leq 1$ and $i(n+1)$ is coprime to 
  char$(\kappa)$, then $\Sha_{div}(F, G) = \{ 1 \}$. 
 \end{cor} 
 
 \begin{proof} Since $G$ is a semisimple simply connected group over $F$ of type $A_n$, 
 there exists a finite extension $E/F$ and a absolutely simple simply connected group $G'$  of type $A_n$
 such that $G$ is the corestriction $R_{E/F}(G')$  (cf., \cite[Theorem 26.8]{KMRT}).  Since 
 $H^1(M,  R_{E/F}(G')) \simeq H^1(M \otimes E, G')$ for any field extension $M/F$ (cf., \cite[Lemma 29.6]{KMRT}), 
 replacing $F$ by $E$ and $G$ by $G'$ we assume that $G$ is absolutely  simple. 
 
 Suppose $G$ is of type $^1A_n$. Then  there exists a central simple algebra $A/F$ 
 of degree $n+1$ such that $G = SL_1(A)$ (cf., \cite[Theorem 26.9]{KMRT}). 
  Hence, by (\ref{sk1-cd1}), $\Sha_{div}(F, G) = \{ 1 \}$.

  Suppose $G$ is of type $^2A_n$. Then  there exists a 
  quadratic field  extension $L/F$ and   a central simple algebra $A$ of degree $n+1$ with a unitary 
 involution $\tau$ such that $G = SU(A, \tau)$   (cf., \cite[Theorem 26.9]{KMRT}). 
  Hence,    by (\ref{sk1u-cd1}), $\Sha_{div}(F, G) = \{ 1 \}$. 
 \end{proof}
%
%
%
%
%
%

 \section{Cohomological dimension 2}
   \label{cd2}
 
 In this section we prove our main theorem when the residue field has cohomological dimension at most 2.

    \begin{theorem} 
    \label{sk1-cd2}
  Let $T$ be a complete discrete valuation ring with residue field $\kappa$ and field of 
fraction $K$. Let $F$ be the function field of a curve over $K$. 
Let $D/K$ be  a central  division algebra  of  period   $n$ coprime to char$(\kappa)$.
If cd$_n(\kappa) \leq 2$, then   $\Sha_{div}(F, SL_1(D))$ is trivial. 
\end{theorem}

\begin{proof}  
Let $\XX_0$ be a regular proper model of $F$ such that    the closed fibre  $X_0$
is a union of regular curves with normal crossings.  Since 
$D$ is defined over $K$, the ramification locus of $D$ is contained in $X_0$. 

Let $t \in K$ be a parameter. Since per$(D)$ is coprime to char$(\kappa)$, we have 
$D = D_0 + (E, \sigma, t)$ for some  unramified central division algebra $D_0$ over $K$
and $E/K$ an unramified cyclic extension.  

Let $\PP_0$ be a finite set of closed points of $\XX$ containing all the singular points of $X_0$
and at least one closed point from each irreducible component of $X_0$.  Let $\UU_0$ be the set
of irreducible components of $X \setminus \PP_0$ and $\BB_0$ the set of branches with respect to $\PP_0$.

Let  $\wp \in \BB$.  Let $\hat{R}_\wp$ be the valuation ring of $F_\wp$.
Then $K \subset F_\wp$ and  $T  \subset \hat{R}_\wp$.   Let  $D_\wp$ be the 
central division algebra over $F_\wp$ which is Brauer equivalent  $D \otimes F_\wp$.
Let $\Lambda$ be the 
unique maximal $\hat{R}_\wp$-order of $D_\wp$. Then we have 
$E \otimes F_\wp = \prod Z(\bar{D}) $ and $\bar{D}$ Brauer equivalent  to 
$\bar{D}_0 \otimes \bar{E}$. In particular  $Z(\bar{D})$ and $\bar{D}$ are unramified.
Hence, by (\ref{sk1d-cd2}), $SK_1(D)(F_\wp)$ is trivial.   
Thus, as in the proof of (\ref{sk1-cd1}),  
  $\Sha_{div}(F, SL_1(D))$ is trivial. 
 \end{proof}

\begin{theorem}
\label{sk1u-cd2}
 Let $T_0$ be a complete discrete valuation ring with residue field $\kappa$ and field of 
fraction $K_0$. Let $K/K_0$ be  a quadratic  \'etale extension and 
  $A/K$    a central simple algebra over $K$ of period $n$   with a $K/K_0$-involution. 
 Let $F_0$ be the function field of a curve over $K_0$ and  $F = F_0K$. 
Suppose that $2n$ is   coprime to char$(\kappa)$. 
Let $\tau$ be  a $F/F_0$-involution on $A$. 
If cd$_n(\kappa) \leq 2$, then   $\Sha_{div}(F_0, SU(A, \tau))$ is trivial.  
 \end{theorem}
 
 \begin{proof}   Let $\XX_0$ be a regular proper model of $F_0$ with the union of 
the  ramification locus of $(A,\tau)$ and the closed fibre  $X_0$
is a union of regular curves with normal crossings. 
Let $\PP_0$ be a finite set of closed points of $\XX$ containing all the singular points of $X_0$
and at least one closed point from each irreducible component of $X_0$. 
Let $\UU_0$ be the set
of irreducible components of $X \setminus \PP_0$ and $\BB_0$ the set of branches with respect to $\PP_0$.

Let  $\wp \in \BB$.  Then $F_{0\wp}$ is a complete discretely  valued field with residue field $\kappa(\wp)$ also a complete 
discretely valued field with residue field a finite extension of $\kappa$. 
As in the proof of (\ref{sk1-cd2}), by using (\ref{sk1u-dvr}),  we get the triviality of $SUK_1(A, \tau)(F_{0\wp})$. 
 
Thus, as in the proof of (\ref{sk1u-cd1}), $\Sha_{div}(F_0, SU(A, \tau))$ is trivial.  
 \end{proof}

 \begin{cor} \label{an-cd2}
 Let $K$ be a complete discretely  valued  field  with residue field $\kappa$. 
  Let $F$ be the function field of a curve over $K$.  Let $G$ be a semisimple simply connected group over $K$
   of type $^iA_n$. If   cd$_{n+1}(\kappa) \leq 2$ and $i(n+1)$ is coprime to 
  char$(\kappa)$, then $\Sha_{div}(F, G) = \{ 1 \}$.
  \end{cor} 
 
 \begin{proof} Using (\ref{sk1-cd2})) and (\ref{sk1u-cd2}), the proof is similar to the proof of (\ref{an-cd1}). 
 \end{proof}

\providecommand{\bysame}{\leavevmode\hbox to3em{\hrulefill}\thinspace}

 \end{document}